%% file: persistence_Lotka_Volterra.tex
\documentclass[11pt]{article}
\usepackage{amsmath,amsfonts,amssymb,latexsym,amsthm,dsfont}
\usepackage[a4paper,margin=2.5cm]{geometry}
\usepackage{graphicx}
\usepackage{booktabs}
\usepackage{hyperref}
\usepackage[utf8]{inputenc}

\usepackage{pgffor}
\usepackage{tikzexternal}
\usepackage{enumitem}
\newtheorem{thm}{Theorem}[section]%
\newtheorem{lem}[thm]{Lemma}%

\newtheorem{defi}[thm]{Definition}%
\newtheorem{ass}[thm]{Assumption}%
\newtheorem{rem}[thm]{Remark}%
\newtheorem{conj}[thm]{Conjecture}%

\newcommand{\DeclareAlphabet}[2]{% l'usine à commandes
  \foreach \x in {A,B,...,Z}{%
    \expandafter\xdef
    \csname #1\x\endcsname{%
      \noexpand#2{\x}}%
  }
}

\DeclareAlphabet{d}{\mathbb}  % \dA -> \mathbb{A}, etc.
\DeclareAlphabet{c}{\mathcal}
\DeclareAlphabet{r}{\mathrm}
\DeclareAlphabet{b}{\mathbf}

\newcommand{\ABS}[1]{{{\left| #1 \right|}}} % |1|

\newcommand{\BRA}[1]{{{\left\{#1\right\}}}} % {1}
\newcommand{\PAR}[1]{{{\left(#1\right)}}} % (1)
\newcommand{\SBRA}[1]{{{\left[#1\right]}}} % [1]
\renewcommand{\leq}{\leqslant}
\renewcommand{\geq}{\geqslant}

\newcommand{\esp}[1]{\mathbb{E}\left[#1\right]}
\newcommand{\var}[1]{\mathrm{Var}\left(#1\right)}

\newcommand{\cvxleq}{\leq_{\mathrm{cvx}}}

\title{On the persistence regime for Lotka-Volterra in randomly fluctuating
environments}
\author{Florent Malrieu, Pierre-André Zitt} %

\date{\today}

\begin{document}

\maketitle
\begin{abstract}
  In this note, we study the long time behavior of Lotka-Volterra systems whose
  coefficients vary randomly. Benaïm and Lobry (2015) recently established that
  randomly switching between two environments that are both favorable to the
  same species may lead to different regimes: extinction of one species or the
  other, or persistence of both species.  Our purpose here is to describe more
  accurately the range of parameters leading to  these regimes, and the support
  of the invariant probability measure in case of persistence. 
\end{abstract}

\section{Introduction}
\subsection{The model}
The aim of the present paper is to study the ergodicity of a piecewise
deterministic Markov process (PDMP) linked to Lotka-Volterra type dynamics.
These lines can be seen as a companion paper to~\cite{lotka} since we go one
step further in the description of different regimes of the process and the support of the
invariant measures. Let us  first provide  an overview of the 
main results in~\cite{lotka} before stating our contribution.

For a given set of positive parameters $\mathcal{E}=(a,b,c,d,\alpha,\beta)$,
consider the Lotka-Volterra differential system in $\dR_+^2$, given by 
\begin{equation}
  \label{eq:lv}
  \begin{cases}
  \dot{x}=\alpha x (1-ax-by) \\
  \dot{y}=\beta y (1-cx-dy) \\
  (x_0,y_0)\in\dR_+^2. 
  \end{cases}
\end{equation}
This system modelizes the evolution of the populations of two species ($x$ of
type~$\textbf{x}$ and $y$ of type~$\textbf{y}$).  The populations
grow logistically --- as encoded by the $\alpha x(1-ax)$ and $\beta y(1-dy)$ terms --- and
compete with each other, which  gives rise to the cross terms $\alpha b xy$
and $\beta c xy$.  We denote by $F_\mathcal{E}$ the associated  vector
field: $(\dot x,\dot y)=F_{\mathcal{E}}(x,y)$. In the sequel, the variable $z$
stands for $(x,y)$ and we will sometimes write $F_{\mathcal{E}}(z)$ instead of
$F_{\mathcal{E}}(x,y)$. This ODE system, taken alone,  is easy to analyze. 
In particular it has only a finite number of equilibrium points, towards which 
the dynamics converges. These equilibria may be on the coordinate axes --- meaning
that one of the species gets extinct --- or inside the positive quadrant. The 
position and nature of the equilibria turn out to depend only on the signs of
$c-a$ and $d-b$. A complete picture  will be given in Section~\ref{sec:notation};
let us note already that when $a<c$ and $b<d$,
the point $(1/a,0)$ attracts any path starting in $(0,+\infty)^2$. We say 
in this case that
the environnement is \emph{favorable to species $\textbf{x}$}\,; it leads to the
extinction of species $\textbf{y}$.    

Consider now two such systems, labeled $0$ and $1$, with environments 
$\mathcal{E}_i=(a_i,b_i,c_i,d_i,\alpha_i,\beta_i)_{i=0,1}$. We make the following
standing assumption: 
\begin{ass} The environments $\cE_0$ and $\cE_1$ are both 
  favorable to $\textbf{x}$: $a_i<c_i$ and $b_i<d_i$ for $i=0,1$. 
  In particular, taken alone, each system leads to the extinction of the second
  species, that is, $y_t$ converges to $0$ and $\limsup x_t>0$ as soon as
  $x_0>0$.  
\end{ass}

Finally, introduce the random process obtained by switching between 
these two deterministic dynamics, at rates $(\lambda_i)_{i=0,1}$. More precisely, 
we consider the process $(Z,I)$ on $\dR^2\times\BRA{0,1}$ driven by the 
infinitesimal generator 
\[
Lf(z,i)=F_{\mathcal{E}_i}(z)\cdot\nabla_z f(z,i) +\lambda_i (f(z,1-i)-f(z,i)).  
\]
In other words, $I$ jumps from $i$ to $1-i$ after a random time with an exponential distribution 
of parameter $\lambda_i$, and while $I_t$ is equal to $i$,
$Z$ evolves deterministically by $\dot Z_t=F_{\mathcal{E}_i}(Z_t)$. The coordinates 
of $Z_t$ are denoted by $X_t$ and $Y_t$. We refer to~\cite{lotka} for a detailed biological 
motivation.  

It is shown in \cite{lotka} that, depending on the environments $\cE_0$, $\cE_1$ 
and the jump rates $\lambda_0$, $\lambda_1$, 
one of the following four things occur almost surely: 
\begin{itemize}
\item extinction of species $\textbf{x}$: $X_t \to 0$ and $\limsup Y_t>0$,
\item extinction of species \textbf{y}: $Y_t\to 0$ and $\limsup X_t>0$, 
\item extinction of one of the two species,  chosen randomly,
\item persistence: the empirical occupation measure (and, in many cases, the
 distribution) of ${(X_t,Y_t)}_{t\geq 0}$ 
converges to a probability measure on $(0,+\infty)^2$ that is absolutely continuous with 
respect to the Lebesgue measure. 
\end{itemize}
Moreover, one or more of these regimes may occur when the jump 
rates $(\lambda_0,\lambda_1)$ vary, the environments $(\mathcal{E}_0,\mathcal{E}_1)$ 
being fixed. Similar surprising behaviors for switched processes have
been studied for linear ODEs in \cite{BLMZ1,mattingly}.

\subsection{The frequent jumps asymptotics and the averaged vector field}
Recall that $\lambda_0$, $\lambda_1$ are the jump rates from one environment
to the other. 
Note that the index process $(I_t)_{t\geq 0}$ is Markov by itself, 
and its invariant measure is a Bernoulli distribution with parameter 
$\lambda_0/(\lambda_0+\lambda_1)$. As a consequence, 
it will be convenient to choose the alternative
parametrization
\begin{equation}\label{eq:para_st}
(s,t)\in [0,1]\times(0,+\infty)\mapsto (st,(1-s)t) 
\end{equation}
for the jump rates, that is, let $t$ be the sum $\lambda_0+\lambda_1$ 
and $s$ be the ratio $\lambda_0/(\lambda_0+\lambda_1)$. 
\begin{rem}[Length of interjump times I]\label{rem:st}
Notice that the expectations of the interjump times are given by $(st)^{-1}$ and $((1-s)t)^{-1}$. 
If $t$ is small, the jumps are rare and the jump times are large in average; as 
$t$ grows the jumps become more frequent and the jump times shorter on average. 
\end{rem}
As the parameter $t$ goes to infinity --- the frequent jumps asymptotics --- it
can be shown that the stochastic process ${(Z_t)}_{t\geq 0}$ converges to the 
solution of 
\[
\dot{z}_t=F_s(z_t) \quad\text{where}\quad 
F_s=(1-s)F_{\mathcal{E}_0}+sF_{\mathcal{E}_1}. 
\]
As noticed in \cite{lotka}, for any $s\in [0,1]$, the vector field 
$F_s$ is the Lotka-Volterra system associated to the "averaged" environment 
$\mathcal{E}_s=(a_s,b_s,c_s,d_s,\alpha_s,\beta_s)$ with 
\begin{align}
\alpha_s &= (1-s)\alpha_0+s\alpha_1,  
&
a_s &= \frac{(1-s)\alpha_0a_0+s\alpha_1a_1}{\alpha_s},
&
b_s &= \frac{(1-s)\alpha_0b_0+s\alpha_1b_1}{\alpha_s},\label{eq:def_alphas}
\\
\beta_s &= (1-s)\beta_0+s\beta_1,
&
c_s &= \frac{(1-s)\beta_0c_0+s\beta_1c_1}{\beta_s},
&
d_s &= \frac{(1-s)\beta_0d_0+s\beta_1d_1}{\beta_s}.\label{eq:def_betas}
\end{align}
Recall that by our standing assumption, $a_i<c_i$ and $b_i<d_i$ for $i=0,1$. 
The key point is that \emph{these inequalities may be reversed} 
for the averaged environment $\cE_s$; in some situations 
$\cE_s$ may even be \emph{unfavorable to species~$\textbf{x}\,$!} 
\begin{defi}[Critical parameter regions]
Two critical (possibly empty) parameter regions are defined in~\cite{lotka} by: 
\begin{equation}\label{eq:def-I}
I=(s_1,s_2)=\BRA{ s\in[0,1], a_s > c_s} 
\quad\text{and}\quad 
J=(s_3,s_4)=\BRA{ s\in[0,1], b_s > d_s}. 
 \end{equation}
\end{defi}
 The fact that $I$ and $J$ are indeed intervals is obvious from the definition
 of $a_s$, $b_s$, $c_s$ and $d_s$ by~\eqref{eq:def_alphas}
 and~\eqref{eq:def_betas}. As we have seen, the relevance of these
 intervals stems from the fact that they correspond to different types
 for the averaged environment $\cE_s$. For example, 
 the vector field $F_s$ always has two stationary points on the coordinate axes, 
 but their nature vary: 
\begin{itemize}
\item the stationary point $(1/a_s,0)$ is a well if $s\notin I$ and a saddle point if $s\in I$, 
\item the stationary point $(0,1/d_s)$ is a saddle point if $s\notin J$ and a well if $s\in J$.  
 \end{itemize}
\begin{rem}\label{rem:adiff}
Notice that if $a_0=a_1$ then the interval $I$ is empty. In the sequel we will focus on the 
case when $a_0\neq a_1$ and without loss of generality we will assume that $a_0<a_1$. 
 \end{rem}
More details and  a complete description of the phase portrait of the Lotka-Volterra system
 will be given below (see Section~\ref{sec:notation}). 

 \subsection{Invasion rates}

It is shown in~\cite{lotka} that the regime of the process $(Z,I)$ is determined by the 
the signs of two quantities, called \emph{invasion rates}, $\Lambda_\textbf{x}$ and 
$\Lambda_\textbf{y}$. These quantities are obtained by considering the 
system where one of the species, say $\textbf{y}$,  is almost extinct: $y\ll 1$. 
In this case, remark two things: 
\begin{itemize}
  \item if the current environment is $i$, the growth rate $\frac{\dot{y}}{y}$ 
    of species $\textbf{y}$ is approximately $\beta_i(1-c_ix)$, by \eqref{eq:lv}; 
  \item while $y$ stays negligible, the evolution of $(X,I)$ looks like 
the dynamics of the Markov process on $[0,\infty)\times\BRA{0,1}$ driven by the 
generator
\[
  L_{\textbf{x}}f(x,i)=\alpha_i x(1-a_ix) \partial_x f(x,i)+\lambda_i(f(x,1-i)-f(x,i)),
\]
which turns out to be ergodic with an invariant measure $\mu_\textbf{x}$ on $[0,\infty)\times\BRA{0,1}$ 
that depends on the jump rates $(\lambda_0,\lambda_1)=(st,(1-s)t)$.
\end{itemize}
Over a long period of time, the growth rate of $\textbf{y}$ is therefore
obtained by averaging $\beta_i(1-c_i x)$ with respect to the 
invariant measure $\mu_\textbf{x}$: this leads to the definition 
of the \emph{invasion rate} $\Lambda_\textbf{y}$ by
\[
\Lambda_\textbf{y}(s,t)=\int_{[0,\infty)\times\{0,1\}} \beta_i(1-c_ix)\mu_{\textbf{x}}(dx,di).
\]
When $\Lambda_\textbf{y}$ is positive (respectively negative) 
species $\textbf{y}$ tends to increase (respectively decrease) from low density. Similarly, 
one can define 
\[
\Lambda_\textbf{x}(s,t)=\int_{[0,\infty)\times\{0,1\}} \alpha_i(1-b_iy)\mu_\textbf{y}(dy,di). 
\]
where $\mu_\textbf{y}$ is the invariant probability measure of the stochastic process on 
$\dR\times\BRA{0,1}$ with generator 
\[
  L_{\textbf{y}}f(y,i)=\beta_i y(1-d_iy) \partial_ yf(y,i)+\lambda_i(f(y,1-i)-f(y,i)),
\]

The main result in~\cite{lotka} ensures that the long time behavior of the Markov 
process ($Z_t,I_t)$ is characterized by the sign of these invasion rates, as summed 
up in the following table (see Theorems~3.1, 3.3, 3.4 and 4.1 of \cite{lotka} for precise 
statements). 

\begin{center}
\begin{tabular}{ccc}
\toprule
  & $\Lambda_\textbf{y}>0$ & $\Lambda_{\textbf{y}}<0 $ \\
\hline 
$\Lambda_\textbf{x}>0$& persistence of the two species & extinction of species $\textbf{y}$\\
\hline
$\Lambda_\textbf{x}<0$ & extinction of species $\textbf{x}$ & random extinction of one of the two species \\
\bottomrule
\end{tabular}
\label{tab:signs_of_lambda12}
\end{center}

\subsection{Our contribution}
In view of the previous result, the study of the model is reduced to 
finding the sign of the invasion rates, depending on the parameters 
of the environment and on the jump rates. To state our results, 
we need to introduce a second parametrization for the jump rates 
$(\lambda_0,\lambda_1)\in (0,+\infty)^2$ slightly different from~\eqref{eq:para_st}:
\[
(u,v)\in [0,1]\times(0,+\infty) \mapsto (\alpha_0 uv, \alpha_1 (1-u)v)
\]
in such a way that 
\begin{equation}\label{eq:def_u_v}
u = \gamma_0/(\gamma_0+\gamma_1)
\quad\text{and}\quad 
v=\gamma_0+\gamma_1
\quad\text{where }\gamma_i = \lambda_i/\alpha_i
\quad\text{for } i=0,1.  
\end{equation}

%$(u,v)\in [0,1]\times(0,+\infty)$ given by $u = \gamma_0/(\gamma_0+\gamma_1)$ 
%	and $v=\gamma_0+\gamma_1$, where $\gamma_i = \lambda_i/\alpha_i$. 

The change of parameters $(u,v)=\xi(s,t)$ is triangular in the sense
that $u$ only depends on $s$: 
\[
(u,v)=\xi(s,t)=\PAR{
  \frac{s\alpha_1}{(1-s)\alpha_0 + s\alpha_1},
  \frac{t}{\alpha_0\alpha_1}( (1-s)\alpha_0 + s\alpha_1)
}.
\]
\begin{rem}[Length of interjump times II]
Notice that the new parameter $v$ is proportional to $t$ when $u$ (or $s$) is fixed. 
As a consequence, as in Remark \ref{rem:st}, the interjump times are short when $v$ 
is large and large when $v$ is small.   
\end{rem}

\begin{defi}[Reparametrized invasion rates]
The invasion rates in the $(u,v)$ coordinates are denoted by
\[
\tilde \Lambda_\emph{\textbf{x}}(u,v)=\Lambda_\emph{\textbf{x}}(\xi^{-1}(u,v))
\quad\text{and}\quad 
\tilde \Lambda_\emph{\textbf{y}}(u,v)=\Lambda_\emph{\textbf{y}}(\xi^{-1}(u,v)). 
\]
Similarly,  $\tilde{I}$ (resp. $\tilde{J}$) is the image of $I$ (resp. $J$) 
for the other parametrization. 
\end{defi}
Note that $\tilde{I}$ and $\tilde{J}$ still are 
(possibly empty) intervals.

\begin{rem}
  The parameter $u$ is already implicitly considered in~\cite{lotka}, where 
  it appears in the computations leading to the explicit conditions for the 
  non-emptyness of $I$ (which are equivalent to the positivity of a second 
  degree polynomial). 
\end{rem}

Our first result is an explicit formula for~$\tilde \Lambda_\textbf{y}$, 
suited both to fast numerical computations and theoretical study. 
\begin{lem}[Expression of $\tilde\Lambda_\textbf{y}$]
  \label{lem:exprLambda2}
  Assume that $a_0<a_1$ (see Remark~\ref{rem:adiff}). Let $(u,v)$ 
  given by~\eqref{eq:def_u_v}. The quantity $\tilde \Lambda_\emph{\textbf{y}}$ can 
  be rewritten as:
  \[
    \tilde\Lambda_\emph{\textbf{y}}(u,v) 
    = \frac{1}{(a_1-a_0) \PAR{\frac{1}{\alpha_0} (1-u) + \frac{1}{\alpha_1} u}}
    \esp{\phi(U_{u,v})},
  \]
  where $\phi:[0,1]\to \dR$ is defined by
  \[
    \phi(y) = \frac{1}{a_0+(a_1-a_0) y}P(a_0+(a_1-a_0) y)
  \]
  for some explicit second degree polynomial $P$, and 
  $U_{u,v}$ is a Beta distributed $\mathrm{Beta}\PAR{uv,(1-u)v}$ random variable. 
    
  Moreover, $\phi$ has the following properties: 
  \begin{itemize}
      \item If $\tilde{I}=\emptyset$ then $\phi$ is nonpositive on $[0,1]$; 
      \item If $\tilde{I}=(u_1,u_2)\neq \emptyset$  then $\phi$ is concave, 
      negative on $(0,u_1)\cup(u_2,1)$  and positive on $\tilde{I}$. 
  \end{itemize}
\end{lem}

Recall that a function $f$ is called quasi convex on $(a,b)$ if its
  level sets $\BRA{f\leq t}$ are convex, that is, $f$ decreases
  on $(a,c)$ and increases on $(c,b)$ for some $c\in[a,b]$. 

    Our main result describes more precisely the region of positivity
    for $\Lambda_\textbf{x}$ and $\Lambda_\textbf{y}$. 

\begin{thm}[Shape of the regions]  \label{thm:mainResult}
  There exists a function $u\mapsto v_\emph{\textbf{y}}(u)$ from $(0,1)\to
  [0,\infty]$, with domain $\tilde{I}$, such that $\tilde \Lambda_\emph{\textbf{y}}(u,v) <0$ 
  when $v<v_\emph{\textbf{y}}(u)$ and $\tilde \Lambda_\emph{\textbf{y}}(u,v) >0$ 
  when $v>v_\emph{\textbf{y}}(u)$. 

  Moreover $v_\emph{\textbf{y}}$ is  quasi-convex, continuous  on its domain $\tilde{I}$, and 
  tends to $+\infty$ on the endpoints of $\tilde{I}$. 

  Similarly, there exists a function $s\mapsto t_\emph{\textbf{y}}(s)\in[0,\infty]$, with
  domain $I$, going to infinity at the endpoints of $I$, such that: 
  \begin{itemize}
	\item $\Lambda_\emph{\textbf{y}}(s,t) < 0$ if $t<t_\emph{\textbf{y}(s)}$, 
	\item $\Lambda_\emph{\textbf{y}}(s,t) = 0$ if $t=t_\emph{\textbf{y}(s)}$, 
	\item $\Lambda_\emph{\textbf{y}}(s,t) > 0 $ if $t>t_\emph{\textbf{y}(s)}$. 
  \end{itemize}

  The same statement holds in the parameters $(s,t)$ for the function
  $(-\Lambda_\emph{\textbf{x}})$ with~$I$ replaced by~$J$ and
  with a critical function $t_\emph{\textbf{x}}(s)$. 
\end{thm}

\begin{rem}\label{rem:convex}
  Numerical computations suggest that both $v_\emph{\textbf{y}}$ 
  and~$t_\emph{\textbf{y}}$ are in fact smooth and convex on $\tilde I$ and $I$ 
  respectively. 
\end{rem}

\begin{rem}
  This result is cited in \cite[Proposition 2.5]{lotka}, since it
  answers a conjecture that appeared in a preprint version of~\cite{lotka}. 
\end{rem}
For an illustration of Theorem~\ref{thm:mainResult} and Remark~\ref{rem:convex}, 
see Figure~\ref{fig:lambda_12}. 

\begin{figure}
  {\centering
  \input{lambda1Et2.tex}

  \bigskip

  \input{lambda1Et2_bis.tex}

}

  {\small
    These plots represent the "critical" functions $t_\textbf{y}$ and 
   $t_\textbf{x}$ for different choices of the environments.  
   Denoting environments by the couple 
   $\PAR{\begin{smallmatrix} \alpha \\ \beta\end{smallmatrix}}$; 
   $\PAR{\begin{smallmatrix} a & b \\ c & d\end{smallmatrix}}$,
   the functions are plotted with
   \[
     \cE_0 = 
   \begin{pmatrix} 1 \\ 5 \end{pmatrix} ; 
   \begin{pmatrix} 1 & 1 \\ 2 & 2\end{pmatrix} 
   \text{ (top plot);}
   \qquad
   \cE_0 = 
   \begin{pmatrix} 1 \\ 2 \end{pmatrix} ; 
   \begin{pmatrix} 1 & 2/3 \\ 2 & 4/3\end{pmatrix} 
   \text{ (bottom plot);}
   \qquad \cE_1 = 
   \begin{pmatrix} 5 \\ 1 \end{pmatrix} ; 
   \begin{pmatrix} 3 & 3 \\ 4 & \rho\end{pmatrix}
 \]
 for various values of the parameter $\rho$ appearing in the definition of the environment $\cE_1$. 
 The black curve in both plots is $t_\textbf{y}$, 
 and does not depend on the value of $\rho$. The colored curves are $t_\textbf{x}$. 
The respective domains of these curves are the intervals~$I$ and~$J$. All configurations
are possible: $I\cap J$ may be empty (bottom plot, $\rho=6.8$), a strict subset of
$I$ and $J$ (bottom plot, $\rho=6.2$) or may be $I$ or $J$ itself (top plot). 

Thanks to the results of \cite{lotka} summarized in the table page
\pageref{tab:signs_of_lambda12}, these plots
describe exactly what regimes are possible when the jump rates (parametrized by 
$s$ and $t$) vary, for a given choice of the environments. 

For example the top plot for $\rho=5.5$ has three regimes: extinction of $\textbf{x}$ 
(above the red curve), persistence (between the red and the black curves) 
and extinction of $\textbf{y}$ (below the black curve). For $\rho=4.5$ there is an 
additonal zone (above the yellow curve and below the black one) of extinction of
a random species. In particular, the knowledge of the relative positions
of $I$ and $J$ is not enough to determine the possible regimes. 

All these plots are computed by finding, for a fixed $s$, the zero of the function
$t\mapsto\Lambda(s,t)$; this is done by a simple root finding 
algorithm, using the explicit formula given in Lemma~\ref{lem:exprLambda2}
to evaluate $\Lambda(s,t)$. 
}

  \caption{Shape of positivity regions for $\Lambda_\textbf{x}$ and $\Lambda_{\textbf{y}}$ }
  \label{fig:lambda_12}
\end{figure}

Finally, our last results are dedicated to the support of the non-trivial invariant probability 
measure in the persistence regime. In~\cite{lotka}, it is shown that this measure has a 
density with respect to the Lebesgue measure on the quadrant. Theorem~\ref{thm:support}
provides a full description of its support when the set $I\cap J$ is not empty. Since a precise 
statement requires several notations introduced in Section~\ref{sec:notation}, we postpone 
it to the last section of the document. 

The remainder of the paper is organized as follows.  In
Section~\ref{sec:notation} we describe the various phase portraits for
Lotka-Volterra vector fields, and narrow down the choices of $\cE_0$, $\cE_1$
that lead to interesting behaviour.  In Section~\ref{sec:proofLemma} we prove
Lemma~\ref{lem:exprLambda2}; the main result is proved in
Section~\ref{sec:proofTheorem}. The final section is dedicated 
to the description of the support of invariant measures in the persistence 
regime.

%%%%%%%%%%%%%
\section{Deterministic picture}\label{sec:notation}
\subsection{Phase portraits of Lotka-Volterra vector field}

We consider here the ODE \eqref{eq:lv} in an environment $\mathcal{E}=(a,b,c,d,\alpha,\beta)$
and describe its possible qualitative behaviours. Much of this description
can be found in~\cite{lotka}, we give it here for the sake of  clarity. 

Barring limit cases that we will not consider, there are essentially
four different phase portraits for the system, that are depicted in Figure~\ref{fi:cases}. 
These four regimes are obtained as follows. 

Notice first that the vector $F_{\mathcal{E}}(x,y)$ is horizontal if $y=0$ or
$cx+dy=1$: we call the line $cx+dy=1$ the horizontal isocline. Similarly
$F_\cE(x,y)$ is vertical if $x=0$ or if $(x,y)$ is on the vertical isocline
$ax+by=1$. These isoclines are the bold straight lines in
Figure~\ref{fi:cases}. 

Each axis is invariant that why in the sequel we are only interested in initial 
conditions with positive coordinates. The three points $(0,0)$, $(0,1/d)$ and 
$(1/a,0)$ are  stationary for $F_{\mathcal{E}}$. The origin is always a source.  
The nature of the other points and the existence of a fourth stationary point 
depends on the parameters; this gives rise to the four types announced above. 

\emph{Type 1.} If $a<c$ and $b<d$,  species $\textbf{y}$ gets extinct.
	\begin{itemize}
\item The point $(1/a,0)$ is the unique global attractor: any solution of the
  ODE starting from a point with positive coordinates converges to $(1/a,0)$. 
\item The point $(0,1/d)$ is a hyperbolic saddle point.  Its stable manifold is
  the vertical axis.  We denote by $\Sigma$ the intersection of its unstable
  manifold with the positive quadrant: $\Sigma$ is the curve made of points $(x_0,y_0)$ such
  that the solution $(x_t,y_t)$ starting at $(x_0,y_0)$ satisfies 
\[
(x_t,y_t)\xrightarrow[t\to-\infty]{}(0,1/d).
\]
We may compute explicitely some characteristics of $\Sigma$: 
in particular, it  leaves $(0,1/d)$ with a slope $-\frac{\alpha (d-b)+\beta d}{\beta c}$
and ends in $(1/a,0)$ with a possibly degenerate slope
$\min\PAR{0,-\frac{\beta(a-c)+\alpha a }{\alpha b}}$. Moreover, $\Sigma$ lies 
in between the horizontal and vertical isoclines. 
\end{itemize}         

\emph{Type 2.} If $a>c$ and $b>d$, species $\textbf{x}$ gets extinct. 
$(0,1/d)$ is the unique sink and $(1/a,0)$ is a saddle point. This is the same
as Type 1 except that the two species $\textbf{x}$ and $\textbf{y}$ are
swapped. 

\emph{Type 3.} If $a>c$ and $b>d$, both species survive.  The points
$(1/a,0)$ and $(0,1/d)$ are saddle points. The isoclines meet at
the sink $(\bar x,\bar y) = (ad-bc)^{-1} (a-c,b-d)$, which 
is the unique global attractor. 

\emph{Type 4.} If $a>c$ and $b<d$, one of the two species gets extinct
depending on the starting point. The meeting point of the isoclines 
$(\bar x,\bar y)=(ad-bc)^{-1} (a-c,d-b)$ is now a saddle point, 
and (non trivial) trajectories converge to one of the
two sinks on the axes, $(1/a,0)$ and $(0,1/d)$. 

\begin{figure}
\begin{center}
 \input{dessins}
\end{center}

 {\small
 In each picture the bold lines are the horizontal and vertical isoclines, 
 the gray lines are trajectories of the~ODE. Sources, sinks and saddle points
 are pictured respectively by white circles, black circles and crosses. 

 Type~$1$ environments correspond to the upper left picture, Type~$2$ to 
 the bottom right one. In Type~$3$ environments (bottom left), the intersection
 of the isoclines attracts the whole quadrant. In Type~$4$ environments (upper right)
 one species or the other gets extinct depending on the starting point. 

 For fixed environments $\cE_0$, $\cE_1$ of Type~$1$, the mixed environment $\cE_s$
 may be of any of the four types, depending on whether~$s$ belongs to the
 intervals~$I$ and~$J$. 
 }
\begin{center}
 \caption{The four possible deterministic regimes for a given environment.}
 \label{fi:cases}
\end{center}
\end{figure}

\subsection{Relative positions of the two environments} 
  
Recall that the vector fields $F_{\mathcal{E}_0}$ and $F_{\mathcal{E}_1}$ are assumed 
favorable to species~\textbf{x}: $a_i<c_i$ and $b_i<d_i$ for $i=0,1$. 
Without loss of generality, we suppose that $a_0<a_1$. The switched 
system may present a surprising behavior if the interval $I$ defined in Equation~\eqref{eq:def-I} 
is not empty. This requires that $c_0<a_1$. As a consequence, we impose 
in the sequel that 
\[
a_0<c_0<a_1<c_1, \quad b_0<d_0 \quad\text{and}\quad b_1<d_1.
\]

\begin{lem}\label{lem:parameters}
If the set $J$ is not empty then $d_1<b_0$ or $d_0<b_1$. Moreover, if $I\cap J$ is non empty 
 then 
\begin{equation}\label{eq:parameters}
a_0<c_0<a_1<c_1 \quad\text{and}\quad b_0<d_0 < b_1<d_1. 
\end{equation}
\end{lem}
\begin{proof}
The first point is clear from the definition of $J$. As a consequence, two configurations 
are compatible with $I$ and $J$ non empty: \eqref{eq:parameters} and   
\[
a_0<c_0<a_1<c_1\quad\text{and}\quad b_1<d_1<b_0<d_0.
\]
Suppose, to derive a  contradiction, that we are in the latter configuration. 
Define the points %$A_i=(1/a_i,0)$, $B_i=(0,1/b_i)$, 
$C_i=(1/c_i,0)$ and $D_i=(0,1/d_i)$ for 
$i=0,1$. If $M=(\tilde x,\tilde y) $ is the intersection of the lines $(C_0D_0)$ 
and $(C_1D_1)$ then the set $[\tilde x,+\infty)\times [0,\tilde y]$ is strongly 
positively invariant under the action by $F_0$ and $F_1$. Furthermore, if $I\cap J$ 
is not empty, this means that there exists $s_0\in (0,1)$ such that $F_{\mathcal{E}_{s_0}}$ is 
of Type 2. In particular, the point $(0,1/d_s)$ is accessible. This is incompatible with the previous 
remark. As a consequence,  $I\cap J$ is empty for this configuration.  
\end{proof}

%%%%%%%%%%%%%%
\section{Expression of the invasion rate}
\label{sec:proofLemma}

Let us first recall the expression of $\Lambda_\textbf{y}$ derived in~\cite{lotka}. 
Letting $p_i = 1/a_i$, $\Lambda_\textbf{y}$ is given by
\[
  \Lambda_\textbf{y}=p_0p_1C\int_{p_0}^{p_1}\! P(x)\theta(x)\,dx  
\]
where 
\begin{align*}
C^{-1} &= \frac{p_1}{\alpha_1}
          \int_{\min(p_0,p_1)}^{\max(p_0,p_1)}
	     \!\frac{\ABS{x-p_0}^{\gamma_0}\ABS{p_1-x}^{\gamma_1-1}}
             {x^{\gamma_0+\gamma_1+1}}\,dx
        + \frac{p_0}{\alpha_0}
	  \int_{\min(p_0,p_1)}^{\max(p_0,p_1)}\!
	    \frac{\ABS{x-p_0}^{\gamma_0-1}\ABS{p_1-x}^{\gamma_1}}
            {x^{\gamma_0+\gamma_1+1}}\,dx,\\
\theta(x)&=
\frac{\ABS{x-p_0}^{\gamma_0-1}\ABS{p_1-x}^{\gamma_1-1}}{x^{\gamma_0+\gamma_1+1}},\\
P(x)&=
\SBRA{\frac{\beta_1}{\alpha_1}(1-c_1x)(1-a_0x)-\frac{\beta_0}{\alpha_0}(1-c_0x)(1-a_1x)}
\frac{a_1-a_0}{\ABS{a_1-a_0}}.
\end{align*}
This quantity is obtained by averaging a growth rate of the second species
with respect to  the invariant measure 
of the one-dimensional PDMP $(X,I)$ on $[0,\infty)\times\BRA{0,1}$ driven by 
\[
  Lf(x,i)=\alpha_i x(1-a_ix) \partial_x f(x,i)+\lambda_i(f(x,1-i)-f(x,i)),
\]
which corresponds to the dynamics of species \textbf{x} when species $\textbf{y}$ is
gone.  In the sequel, we assume that $a_0<a_1$ and set $\delta = a_1 - a_0$.
It is obvious that the recurrent set of $(X,I)$ is $[p_1,p_0]\times\BRA{0,1}$.

In~\cite{lotka}, the continuous part $[p_1,p_0]$ of the state space 
is parametrized by a mapping $[0,1]\ni s\mapsto p_s$
 given by $p_s = 1/a_s$, where
$a_s = (s\alpha_1 a_1 + (1-s)\alpha_0 a_0) / (s\alpha_1 + (1-s)\alpha_0)$. 
It is interesting to look at another
parametrization $[0,1]\ni u \mapsto \tilde{p}(u)$, where
$\tilde{p}(u) = \tilde{a}(u)^{-1} = \frac{1}{a_0 + \delta u}$. 
These parametrizations are summed up in the following diagram:

\begin{center}
\begin{tikzpicture}
  \matrix (m) [matrix of math nodes,column sep=8em,row sep = 3ex]{
     & {[a_0,a_1]} &  \\
     I\subset [0,1] & & {[0,1] \supset \tilde{I} }\\
     & {[p_1,p_0]} & \\
    };
    \draw[->] (m-2-1)-- node[auto]      {$s\mapsto a_s$}          (m-1-2);
    \draw[->] (m-2-3)-- node[auto,swap] {$u\mapsto \tilde{a}(u)$} (m-1-2);
    \draw[->] (m-2-1)-- node[auto,swap] {$s\mapsto p_s$}          (m-3-2);
    \draw[->] (m-2-3)-- node[auto] {$u\mapsto \tilde{p}(u)$} (m-3-2);
    \draw[->] (m-1-2)-- node[auto] {$x\mapsto 1/x$}          (m-3-2);
\end{tikzpicture}
\end{center}

This parameter $u$ is  the one given in the introduction
and
corresponds
to a ratio of the $\gamma$, when $s$ corresponds to a ratio of $\lambda$, 
in the sense that:
\[ 
  \tilde{p}\PAR{\frac{\gamma_0}{\gamma_0 + \gamma_1}} = p\PAR{\frac{\lambda_0}{\lambda_0+
\lambda_1}}.
\]
\begin{rem}
  As already mentioned above, the parameter $u$ and the interval $\tilde{I}$
  are used implicitly in~\cite{lotka}: $u$ appears in Remark 1, and the map $S$
  defined at the beginning of Section $4$ is given in our notation by $S(u) =
  p^{-1}(\tilde{p}(u))$. 
\end{rem}

Let us study the integral $\int_{p_1}^{p_0} P(x)\theta(x)dx$. 
Set  $y=\tilde{p}^{-1}(x)$, so that:
\begin{align*}
  x&= \tilde{p}(y) = \frac{1}{\tilde{a}(y)} = \frac{1}{a_0+\delta y}, &
  dx &= -\delta \tilde{p}(y)^2 dy \\
  p_0 - x &= \delta p_0  y\tilde{p}(y), &
  x - p_1 &=  \delta p_1 (1-y)\tilde{p}(y).
\end{align*}
Changing variables in the integral yields:
\begin{align*}
  \int_{p_1}^{p_0} P(x) \theta(x) dx 
  &= \int_0^1 P(\tilde{p}(y))\PAR{\delta p_0 y \tilde{p}(y)}^{\gamma_0 - 1}
  \PAR{\delta p_1 (1-y) \tilde{p}(y)}^{\gamma_1 - 1} 
  \tilde{p}(y)^{-\gamma_0 - \gamma_1 - 1}
  \delta \tilde{p}(y)^2 dy \\
  &= 
  \delta^{\gamma_0+\gamma_1 - 1} p_0^{\gamma_0-1}p_1^{\gamma_1 - 1}
  \int_0^1 P(\tilde{p}(y)) \frac{1}{\tilde{p}(y)} y^{\gamma_0 -1} (1-y)^{\gamma_1 - 1} dy \\
  &= 
  \delta^{\gamma_0+\gamma_1 - 1} p_0^{\gamma_0-1}p_1^{\gamma_1 - 1}
  \int_0^1 \phi(y) y^{\gamma_0 -1} (1-y)^{\gamma_1 - 1} dy \\
  &= 
  \delta^{\gamma_0+\gamma_1 - 1} p_0^{\gamma_0-1}p_1^{\gamma_1 - 1}
  B(uv,(1-u)v) \esp{\phi(U_{u,v})}.
\end{align*}
since $\phi(y) = \frac{1}{\tilde{p}(y)} P (\tilde{p}(y))$. 
A similar computation gives the exact formula
\[ C^{-1} = \PAR{ \delta^{\gamma_0+\gamma_1 - 1 } p_0^{\gamma_0 - 1} p_1^{\gamma_1 - 1}
  B(uv,(1-u)v) 
}p_0p_1\delta \PAR{ \frac{1}{\alpha_0}(1- u) + \frac{1}{\alpha_1} u} 
\]
for the normalization constant $C$. Therefore
\begin{align*}
  \Lambda_\textbf{y}(\gamma_0,\gamma_1) 
  &= \frac{1}{\delta\PAR{\frac{1}{\alpha_0} (1-u) + \frac{1}{\alpha_1} u}}
  \esp{\phi(U_{u,v})}.
\end{align*}

Let us study $\phi$ more precisely. Since $P$ is a second-degree polynomial, 
let us write it down as $P(x)=A_2x^2 + A_1x + A_0$. Then 
\[ \phi(y) = \frac{A_2}{a_0 + \delta y} + A_1 + A_0(a_0 + \delta y).\]
The second derivative is 
\[ \phi''(y) = \frac{2A_2\delta^2}{(a_0 + \delta y)^3},\]
which has the sign of $A_2$ on $[0,1]$, so $\phi$ is either strictly
convex or strictly concave. However, the proof of the first item
of Proposition 2.2 in~\cite{lotka} shows that (still in the case $a_0<a_1$)
$P(p_s) = \frac{\beta_s}{\alpha_1 s} (1 - a_0/a_s)(1 - c_s/a_s)$ 
has the same sign as $a_s-c_s$, that is, $P(p_s)$ is positive iff $s\in I$. 
If $I$ is empty, so is $\tilde{I}$, and $\phi$ is negative on $]0,1[$. 
If $I$ is not empty, so is $\tilde{I}$, and $\phi$ is positive on $\tilde{I}$
(and nonpositive outside $\tilde{I}$), therefore $\phi$ must be concave.

\section{Shape of the positivity region}
\label{sec:proofTheorem}

We begin with a lemma, which is proved in the next two sections.  The existence
and  properties of $v_c$ and $t_c$ stated in Theorem~\ref{thm:mainResult} are
deduced from this lemma in Section~\ref{sec:reg}. 
\begin{lem}
  If $I$ is non empty, the map $(u,v)\mapsto \esp{\phi(U_{u,v})}$ is (strictly)
  increasing in $v$ and concave in~$u$. 
\end{lem}

\subsection{\texorpdfstring{Monotonicity in $v$}{Monotonicity in v}}
We wish to compare $\esp{\phi(U_{u,v})}$ for different values of $v$.  Since
the function $\psi = (-\phi)$ is convex, a natural idea is to compare the distributions
of $U_{u,v}$ for various $v$ in the \emph{convex order}. Let us first recall
the definition of this order. 
\begin{defi}[Convex order]
  Let $X$ and $Y$ be two random variables. If the inequality 
  \[ \esp{\psi(X)} \leq \esp{\psi(Y)}\]
  holds for all convex functions $\psi$ such that the expectations exist,
  $X$ is said to be \emph{smaller than $Y$ in the convex order}. 
  We denote by $X\cvxleq Y$ this relation. 
\end{defi}
The convex order admits the following  characterization
in terms of cumulative distribution functions (\cite[Theorem 3.A.1]{SS07}). 
\begin{thm}[Convex order and distribution functions]
  \label{thm:convex_order_cdf}
  The variables $X$ and $Y$ satisfy $X\cvxleq Y$ if and only if
  $\esp{X} = \esp{Y}$ and, for all $x$, 
  \begin{equation}
    \label{eq:ssd}
    \int_{-\infty}^x F_X(t)dt \leq \int_{-\infty}^x F_Y(t) dt < \infty,
  \end{equation}
  where $F_X$ and $F_Y$ are the cumulative distribution functions
  of $X$ and $Y$. 
\end{thm}
\begin{rem}[Terminology]
  The convex order is one possible way of formalizing the idea
  that $Y$ is "more variable/more spread out" than $X$. Note 
  in particular if $X\cvxleq Y$ then $\var{X}\leq \var{Y}$. 
  For details on the convex order,  a survey of other formalizations of
  variability  and many other notions of stochastic order we refer
  to~\cite{SS07}. 

  Let us also note that, when \eqref{eq:ssd} holds (without assuming
  $\esp{X} = \esp{Y}$), $X$ is said to \emph{second-order stochastically dominate
    $Y$}; this is equivalent to asking the inequality $\esp{\psi(X)}\leq\esp{\psi(Y)}$
  for any \emph{non-increasing} convex function $\psi$. We refer to 
  \cite[Appendix B.19]{MOA11} for a proof,  additional discussion, 
  and references to the literature. 
\end{rem}
We will need the following easy fact. 
\begin{thm}
  \label{thm:ordre_strict}
  If $X\cvxleq Y$ and $\esp{\psi(X)} = \esp{\psi(Y)}$ for some 
  strongly convex $\psi$, then $X$ and $Y$ have the same distribution. 
\end{thm}
\begin{proof}
  Suppose that $X$ and $Y$ satisfy the hypotheses. 
  By a classical characterization of the convex order
  (\cite[Theorem 3.A.4]{SS07})
  there exists a couple $(X,Z)$ such that $\esp{Z|X} = 0$ and $X+Z$ 
  has the same distribution as $Y$. Since $\psi$ is strongly convex, 
  there exists an $m>0$ such that for all $t\in[0,1]$, 
  \[
    \psi(X+tZ) =  \psi( (1-t)X + t(X+Z))\leq (1-t)\psi(X) + t\psi(X+Z) - \frac{mt(1-t)}{2}Z^2. 
  \]
  Taking expectations we get
  \[
    \esp{\psi(X)} \leq \esp{\psi(X+tZ)} \leq (1-t)\esp{\psi(X)} + t\esp{\psi(Y)} 
    - \frac{mt(1-t)}{2} \esp{Z^2},
  \]
  where the first inequality comes from Jensen's inequality and $\esp{Z|X}=0$. 
  Since $\esp{\psi(Y)} = \esp{\psi(X)}$, $Z$ must be zero almost surely, 
  so $X$ and $Y$ have the same distribution. 
\end{proof}

\begin{thm}[Orderings between Beta r.v.]
  Let $X\sim \mathrm{Beta}(\alpha,\beta)$ and $X'\sim\mathrm{Beta}(\alpha',\beta')$. 

  If $\alpha< \alpha'$, $\beta<\beta'$ and $\alpha/(\alpha+\beta)  =  \alpha'/(\alpha'+\beta')$, 
  then $X' \cvxleq X$. 

  %If $\alpha < \alpha'$ and $\alpha+\beta = \alpha'+\beta'$, then $X\stoleq X'$. 
\end{thm}
\begin{proof}
  Call $f_{\alpha,\beta}$, $f_{\alpha',\beta'}$ the densities of the distributions. 
  Compute their ratio:
  \[ \frac{f_{\alpha',\beta'}(x)}{f_{\alpha,\beta}(x)} = C_{\alpha,\beta,\alpha',\beta'} x^{\alpha'-\alpha} (1-
x)^{\beta'-\beta}.\]
  In the first case, 
  this ratio starts and ends in zero, is strictly increasing on $[0,x_0]$ and
  strictly decreasing on $[x_0,1]$. Since the
  two functions are densities, the ratio must cross $1$ exactly twice, say in $x_1$, $x_2$.  Therefore
  \[
  d(x) = f_{\alpha',\beta'} - f_{\alpha,\beta}
  \]
  is positive on $(x_1,x_2)$ and negative on $(0,x_1)$ and $(x_2,1)$. 
  Therefore
  \[ D(x) = F_{X'}(x) - F_{X}(x)\]
  starts at zero, decreases on $[0,x_1]$, increases on $[x_1,x_2]$ and decreases on $[x_2,1]$, so
  $D(x)$ is negative on $[0,x_3]$ and positive on $[x_3,1]$ (since it ends at zero). 
  Integrating once more, 
  \[ \int_0^x D(t) dt\]
  starts and ends at zero (since $\esp{X} = \esp{X'}$) and is decreasing-increasing, therefore it is non-
positive. Thanks to Theorem~\ref{thm:convex_order_cdf}, this implies $X'\cvxleq X$. 
\end{proof}

\begin{proof}[Proof of the monotonicity in $v$]
   Suppose $v<v'$ and put $\alpha= uv$, $\alpha'=uv'$, $\beta=(1-u)v$, $\beta'=(1-u)v'$. The
   theorem shows that $U_{u,v'}\cvxleq U_{u,v}$. Since 
   $\psi = (-\phi)$ is strongly convex, this implies
   \[ \esp{\phi(U_{u,v})} < \esp{\phi(U_{u,v'})}, \]
   by Theorem~\ref{thm:ordre_strict},
   so $v\mapsto \esp{\phi(V_{u,v})}$ is (strictly) increasing. 
\end{proof}

\subsection{\texorpdfstring{Concavity in $u$}{Concavity in u}}
Even though $\phi$ is concave and the r.v. $U_{u,v}$, $U_{u',v}$ are
stochastically comparable (for the usual stochastic order), this is not enough
to show the concavity of $u\mapsto \esp{\phi(U_{u,v})}$. 
We prove it by elementary means. First, recalling the explicit expression
of $\phi$ we write:
\begin{align*}
  \esp{\phi(U_{u,v})} 
  &= A\esp{\frac{1}{a_0 + \delta U_{u,v}}} + B + C \esp{a_0 + \delta U_{u,v}} \\
  &= A\esp{\frac{1}{a_0 + \delta U_{u,v}}} + B + C(a_0 + \delta u). 
\end{align*}
Since $A$ is negative (see the proof of Lemma~\ref{lem:exprLambda2}) and
the last term is linear, it is enough to show that 
\[
  g: u \mapsto \esp{\frac{1}{a_0+\delta U_{u,v}}}
\]
is convex. It is a bit easier to use the symmetry of Beta random 
variables
($U_{u,v}$ and $1-U_{1-u,v}$ have the same distribution) and look at
\[ 
  h: u\mapsto \esp{\frac{1}{a_1 - \delta U_{u,v}}}. 
\]
Since $g(u)=h(1-u)$, $g$ will be convex if $h$ is convex. 
Now, recalling that $\delta = (a_1-a_0)< a_1$, we write
the series development:
\begin{align*}
  h(u) = \frac{1}{a_1} \esp{\sum_{k=0}^\infty \PAR{\frac{\delta}{a_1}}^k U_{u,v}^k} 
  = \frac{1}{a_1} \sum_{k=0}^\infty           \PAR{\frac{\delta}{a_1}}^k \esp{U_{u,v}^k}
\end{align*}
so $h$ is a mixture of the functions $h_k$ given by:
\[
  h_k(u) =  \esp{U_{u,v}^k} 
  = \prod_{r=0}^{k-1} \frac{ uv + r}{v+r}
  = \frac{1}{\prod_{r=0}^{k-1}(v+r)} (uv)(uv+1)\cdots (uv + k-1). 
\]
This is a polynomial function of $u$ with positive coefficients, therefore
it is convex on $[0,1]$. This concludes the proof. 

\subsection{Properties of the frontier}
\label{sec:reg}
We have just shown that $v\mapsto \tilde \Lambda_\textbf{y}(u,v)$ is  
increasing.  From~\cite[Proposition 2.2]{lotka} and the fact that in the change
of variables $(s,t)\leftrightarrow (u,v)$, $u$ only depends on $s$, and for
fixed $u$, $v$ is an increasing function of $t$,  we know additionally that it
admits:
\begin{itemize}
  \item negative limits at $0$ and $\infty$ if $u$ does not belong to the closure of $\tilde{I}$, 
  \item a negative limit at $0$ and a positive limit
    at $\infty$ if $u\in \tilde{I}$. 
\end{itemize}
Since $\tilde \Lambda_\textbf{y}$ has the same sign as $\esp{\phi(U_{u,v})}$, 
this justifies the existence of $v_c$, and we have
\[ 
  \tilde \Lambda_\textbf{y}(u,v) = 0 \iff u\in\tilde{I}, v=v_c(u). 
\]

Let us prove that $v_c$ is quasi-convex. 
Let $a<b<c$ be three points in $(0,1)$. If $a$ or $c$ are not in $\tilde{I}$, 
$v_c(b)$ is less than $\max(v_c(a), v_c(c)) = \infty$. If $a$ and $c$ are
in $\tilde{I}$, let $M>\max(v_c(a), v_c(c))$. Since $\tilde \Lambda_\textbf{y}(u,\cdot)$ 
is increasing, $\tilde \Lambda_\textbf{y}(a,M)$ and $\tilde \Lambda_\textbf{y}(c,M)$ are positive. 
Since $u\mapsto \dE\phi(U_{u,v})$ is concave, $\tilde \Lambda_\textbf{y}(b,M)$ is positive. 
Therefore $v_c(c) \leq M$. Sending $M$ to $\max(v_c(a), v_c(b))$ yields the 
quasi-convexity of $v_c$. 

Let us now show the regularity properties. Let $u_n$ be an increasing 
sequence in $\tilde{I}$, converging to some $u\in(0,1)$. Since
$v_c$ is quasi-convex, $v_n = v_c(u_n)$ is eventually  monotone, 
so it converges to some  $v\in [0,\infty]$. If $v$ is finite, since 
the zero set of $\Lambda_\textbf{y}$ is closed, by continuity, $v = v_c(u)$, so 
$u$ must be in $\tilde{I}$. Conversely, 
if $u\in \tilde{I}$, $v_c$ is bounded on
a neighborhood of $u$ by quasi-convexity, so $v$ is finite. This shows
that $v_c$ is continuous on $\tilde{I}$ and converges to $\infty$ 
at the endpoints. 

The properties of the change of variables $(s,t)\leftrightarrow(u,v)$ 
show that $v_c$ is well-defined and continuous with the correct limits. 

\section{Support of the invariant measure} \label{se:support}

Note that the stochastic Lotka-Volterra process has at least two invariant 
probability measures, supported on the coordinate axes. In the persistence regime, 
we are interested in the third invariant measure, whose support 
$\Gamma\times \BRA{0,1}$ is such that $\Gamma$ 
has non empty interior. Several properties of $\Gamma$ are 
established in~\cite{lotka} (see below). In this section, we aim at providing 
a full description of $\Gamma$. Its shape essentially depends 
on the fact that $I\cap J$ is empty or not. 

\subsection{Persistence with "full support"}

In this subsection, we assume that $I\cap J$ is not empty. According to 
Lemma~\ref{lem:parameters}, the vector fields $F_{\mathcal{E}_0}$ and $F_{\mathcal{E}_1}$
are such that \eqref{eq:parameters} holds.
Let us denote by $\Sigma_i$ the  intersection of $[0,\infty)^2$ and the
unstable manifold of $(0,1/d_i)$ and $\Gamma'$ the bounded subset of
$[0,+\infty)^2$ with border 
\[
\Sigma_1\cup\BRA{(x,0)\,:\,1/a_1\leq x\leq 1/a_0}
\cup\Sigma_0\cup\BRA{(0,y)\,:\,1/d_1\leq x\leq 1/d_0}.
\]

\begin{thm}\label{thm:support}
Suppose that $I\cap J\neq \emptyset$. Then, for any $(s,t)\in [0,1]\times
(0,\infty)$ such that $\Lambda_\emph{\textbf{x}}(s,t)>0$ and
$\Lambda_\emph{\textbf{y}}(s,t)>0$, then $\Gamma'=\Gamma$. 
\end{thm}

\begin{figure}
  {
  \centering
  \input{proof_support.tex}
  \par
}
\small
The isoclines (straight lines) and unstable manifolds (curved lines)
for the three environments $\cE_0$ (bottom left, in blue), $\cE_s$ (middle,
in purple) and $\cE_1$ (upper right, in red). Note how the isoclines 
are "swapped" for $\cE_s$, a Type 2 environment. 
\caption{Full support case: isoclines and unstable manifolds}
  \label{fig:IinterJ_non_vide}
\end{figure}

\begin{figure}
  {\centering
    \includegraphics[width=10cm]{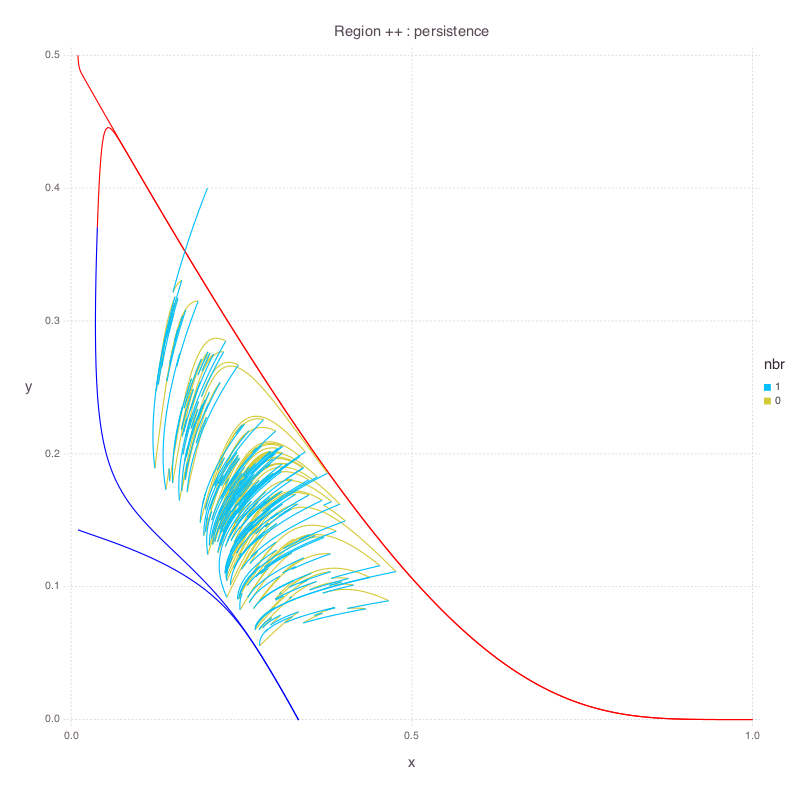}
   \par
  }

  \small
  The outer curves are $\Sigma_0$ and $\Sigma_1$. The region 
  between these curves is positively invariant. The
  inner curves are the two trajectories coming from the 
  unique point $z\in T$: they form the boundary of the support. 
  The sample trajectory shows that the invariant measure is 
  in practice often concentrated on a smaller subset. 
   \caption{Support away from the $y$ axis}
   \label{fig:away}
\end{figure}

\begin{proof}
Firstly, notice that the set $\Gamma'$ is positively invariant for each flow 
since both vector fields $F_{\mathcal{E}_0}$ and $F_{\mathcal{E}_1}$ 
point inside $\Gamma'$. 

Pick an $s\in I\cap J$. The isoclines and the unstable manifold of the 
saddle point for the three environments $\cE_0$, 
$\cE_1$ and $\cE_s$ are necessarily in the position depicted in
Figure~\ref{fig:IinterJ_non_vide}. Denote by $\Sigma_s$ the 
intersection of the unstable manifold of $(1/a_s,0)$ with the
upper right quadrant. 

First step: the set $\Sigma_s$ is contained in the support. 
Indeed, pick a point $(x,y)$ in the interior of the support (such
a point exists by \cite[Remark 6]{lotka}). The loop formed 
by the trajectories starting from $(x,y)$ with both flows (converging
to $A_1:(1/a_1,0)$ and $A_0:(1/a_0,0)$ and the line segment $[A_0,A_1]$
is included in the support (by positive invariance). As a consequence, 
the support must contain a closed half ball centered on $A_s$ --- 
let us call it $\cB$. Now pick a point $(x,y)\in \Sigma_s$: by definition its $\cE_s$ 
flow converges for $t\to-\infty$ to $1/a_s$. By continuity there exists a 
point in the past of $(x,y)$ which is in $\cB$. Running the time forward again, 
the point $(x,y)$ must be in the support. 

Second step: any point (strictly) between $\Sigma_1$ and $\Sigma_s$ is in $\Gamma$. 
Starting from such a point $(x,y)$, run the $\cE_1$ flow in reverse time. 
The trajectory must cross $\Sigma_s$. So $(x,y)$ is in the 
future of a point in $\Sigma_s\subset\Gamma$, and $(x,y)\in\Gamma$ by positive invariance. 

Third step: any point between $\Sigma_0$ and $\Sigma_s$ is in $\Gamma$. 
This step is similar to the previous one and is omitted. 

Similarly, any point between $\Sigma_1$ and $\Sigma_s$ is in $\Gamma$.
\end{proof}

\subsection{Support away from the $y$ axis}

We suppose in the sequel that $I\cap J$ is empty. Let us introduce the 
set where the two vector fields $F_0$ and $F_1$ are collinear:
\[
C=\BRA{z\in\dR_+^2\, : \, \det(F_0(z),F_1(z))=0}.
\]
This set is the union of $\BRA{(0,y)\,:\, y\geq 0}$, $\BRA{(x,0)\,:\, x\geq 0}$, and 
\[
\tilde C=\BRA{(x,y)\in\dR_+^2\, : \,G(x,y)=0}
\]
where $G$ is a polynomial of degree 2. As a consequence, the set~$\tilde C$ is 
a subset of a conic. It is easy to see that $\tilde{C}$ is also the 
set of non-degenerate equilibrium points for the vector field $F_{\cE_s}$, as
$s$ varies from $0$ to $1$. When $s\in I$, $\cE_s$ is of Type~$3$ so the
equilibrium point is stable and globally attractive. Therefore the 
part of $\tilde{C}$ that corresponds to $s\in I$ must be included in $\Gamma$, as
well as all trajectories (for both flows) starting from it. 

Numerical experiments suggest that there is a unique "extremal point" 
on this part of $\tilde{C}$, such that the trajectories starting from 
this point form the boundary of $\Gamma$. See Figure~\ref{fig:away}. 

To describe it more precisely, consider the subset of $\tilde C$ made of the points 
where $F_0$ (or $F_1$) is tangent to the curve $\tilde C$. This set is given by 
\[
T=\BRA{(x,y)\in \dR_+^2\,:\, G(x,y)=0\text{ and } (F_0\cdot \nabla G)(x,y)=0}.
\]
Since $G$ and $F_0\cdot\nabla G$ are polynomials with respective degrees 2 and 3, $T$ 
is made of at most six points according to Bezout's Theorem. 

%For any $z\in T$ and $i=0,1$, let us define 
%\[
%\Sigma_i(z)=\BRA{\varphi^{i}_z(t)\, :\, t\in [0,\infty)}.
%\]
%For any $z\in T$, $C(z)=\Sigma_0(z)\cup\Sigma_1(z)\cup [1/a_1,1/a_0]\times \BRA{0}$ 
%is a Jordan curve.    

For any $z\in T$ let us define $C(z)$ the bounded region enclosed
by the Jordan curve
\[
\BRA{\varphi^{0}_z(t)\, :\, t\in [0,\infty)}\cup \BRA{\varphi^{1}_z(t)\, :\, t\in [0,\infty)} 
\cup [1/a_1,1/a_0]\times \BRA{0}, 
\]
where $t\mapsto \varphi^{i}_z(t)$ is the flow associated to the vector field $F_i$ for $i=0,1$.    

\begin{conj}
The set $T$ is a singleton $\BRA{z_0}$ and the support of the invariant measure 
which is not supported by one of the two axes is $C(z_0)\times\BRA{0,1}$. 
\end{conj}

%\begin{proof}
%  First case: $\# T = 1$ and $J$ empty. Only one piece of conic. Run
%  the red and green trajectories starting from the unique point of $T$. 
%  They do not intersect the conic again (the scalar product between red 
%  and the tangent along the conic has at most one zero, so it cannot change
%  signe more than once etc.) A duck feet argument concludes. 
%
%  Second case: $\#T = 1$ and $J$ non empty. \FIX prove that the trajectories
%  cannot go up and cross the upper part of the conic. 
%\end{proof}

\paragraph*{Acknowledgements.}
We thank an anonymous referee for constructive remarks.
We acknowledge financial support from the French ANR project ANR-12-JS01-0006-PIECE.
Numerical computations were done in Julia and graphics in TikZ.

\bibliography{biblio}
\bibliographystyle{abbrv.bst}

{\footnotesize

\noindent Florent \textsc{Malrieu}, 
e-mail: \texttt{florent.malrieu(AT)univ-tours.fr}

\medskip

\noindent\textsc{Laboratoire de Math\'ematiques et Physique Th\'eorique 
(UMR CNRS 7350), F\'ed\'eration Denis Poisson (FR CNRS 2964), Universit\'e 
Fran\c cois-Rabelais, Parc de Grandmont, 37200 Tours, France.}

\bigskip

   \noindent Pierre-Andr\'e~\textsc{Zitt},
e-mail: \texttt{pierre-andre.zitt(AT)u-pem.fr}

 \medskip

 \noindent\textsc{Laboratoire d'Analyse et Math\'ematiques Appliqu\'ees (UMR CNRS 8050), Universit\'e-Paris-Est-Marne-La-Vall\'ee, 
 5, boulevard Descartes,
 Cité Descartes, Champs-sur-Marne,
 77454 Marne-la-Vallée Cedex 2, France.}

}

\end{document}

%% file: lambda1Et2.tex
\begin{tikzpicture}[x=1mm,y=-1mm]
\definecolor{mycolor000000}{rgb}{0,0,0}
\definecolor{mycolorD4CA3A}{rgb}{0.83,0.79,0.23}
\definecolor{mycolor00BFFF}{rgb}{0,0.75,1}
\definecolor{mycolorD0D0E0}{rgb}{0.82,0.82,0.88}
\definecolor{mycolor564A55}{rgb}{0.34,0.29,0.33}
\definecolor{mycolor362A35}{rgb}{0.21,0.16,0.21}
\definecolor{mycolor000000}{rgb}{0,0,0}
\definecolor{mycolor6C606B}{rgb}{0.42,0.38,0.42}
\definecolor{mycolorFF5EA0}{rgb}{1,0.37,0.63}
\definecolor{mycolor00E5B2}{rgb}{0,0.9,0.7}
\definecolor{mycolor4C404B}{rgb}{0.3,0.25,0.29}
\begin{scope}
\begin{scope}
\draw (80.56,68.39) node [text=mycolor564A55,draw=mycolor000000,draw opacity=0,rotate around={-0: (0,1.81)},inner sep=0.0]{\fontsize{3.88mm}{4.66mm}\selectfont $\text{s}$};
\end{scope}
\begin{scope}
\draw (19.34,61.72) node [text=mycolor6C606B,rotate around={-0: (61.22,1.34)},inner sep=0.0]{\fontsize{2.82mm}{3.39mm}\selectfont $\text{0.0}$};
\draw (43.83,61.72) node [text=mycolor6C606B,rotate around={-0: (36.73,1.34)},inner sep=0.0]{\fontsize{2.82mm}{3.39mm}\selectfont $\text{0.2}$};
\draw (68.31,61.72) node [text=mycolor6C606B,rotate around={-0: (12.24,1.34)},inner sep=0.0]{\fontsize{2.82mm}{3.39mm}\selectfont $\text{0.4}$};
\draw (92.8,61.72) node [text=mycolor6C606B,rotate around={-0: (-12.24,1.34)},inner sep=0.0]{\fontsize{2.82mm}{3.39mm}\selectfont $\text{0.6}$};
\draw (117.28,61.72) node [text=mycolor6C606B,rotate around={-0: (-36.73,1.34)},inner sep=0.0]{\fontsize{2.82mm}{3.39mm}\selectfont $\text{0.8}$};
\draw (141.77,61.72) node [text=mycolor6C606B,rotate around={-0: (-61.22,1.34)},inner sep=0.0]{\fontsize{2.82mm}{3.39mm}\selectfont $\text{1.0}$};
\end{scope}
\begin{scope}
\begin{scope}
\draw (147.58,29.23) node [text=mycolor4C404B,rotate around={-0: (-0.06,3.63)},right,inner sep=0.0]{\fontsize{2.82mm}{3.39mm}\selectfont $\text{3.5}$};
\draw (147.58,32.86) node [text=mycolor4C404B,rotate around={-0: (-0.06,0)},right,inner sep=0.0]{\fontsize{2.82mm}{3.39mm}\selectfont $\text{4.5}$};
\draw (147.58,36.48) node [text=mycolor4C404B,rotate around={-0: (-0.06,-3.63)},right,inner sep=0.0]{\fontsize{2.82mm}{3.39mm}\selectfont $\text{5.5}$};
\draw (147.58,40.11) node [text=mycolor4C404B,rotate around={-0: (-0.06,-7.25)},right,inner sep=0.0]{\fontsize{2.82mm}{3.39mm}\selectfont $\text{7.0}$};
\end{scope}
\begin{scope}
\path [fill=mycolor00BFFF,draw=mycolor000000,draw opacity=0] (144.77,28.32) rectangle +(1.81,1.81);
\path [fill=mycolorD4CA3A,draw=mycolor000000,draw opacity=0] (144.77,31.95) rectangle +(1.81,1.81);
\path [fill=mycolorFF5EA0,draw=mycolor000000,draw opacity=0] (144.77,35.58) rectangle +(1.81,1.81);
\path [fill=mycolor00E5B2,draw=mycolor000000,draw opacity=0] (144.77,39.2) rectangle +(1.81,1.81);
\end{scope}
\begin{scope}
\draw (144.77,25.41) node [text=mycolor362A35,draw=mycolor000000,draw opacity=0,rotate around={-0: (5.61,0.19)},right,inner sep=0.0]{\fontsize{3.88mm}{4.66mm}\selectfont $\text{$\rho$}$};
\end{scope}
\end{scope}
\begin{scope}
\clip  (17.34,5) -- (143.77,5) -- (143.77,60.72) -- (17.34,60.72);
\begin{scope}
\clip  (17.34,5) -- (143.77,5) -- (143.77,60.72) -- (17.34,60.72);
\path [fill=mycolor000000,fill opacity=0,draw=mycolor000000,draw opacity=0] (17.34,5) rectangle +(126.43,55.72);
\end{scope}
\begin{scope}
[dash pattern=on 0.5mm off 0.5mm,line width=0.2mm]
\path [fill=mycolor000000,draw=mycolorD0D0E0]  (17.34,58.72) -- (143.77,58.72);
\path [fill=mycolor000000,draw=mycolorD0D0E0]  (17.34,50.1) -- (143.77,50.1);
\path [fill=mycolor000000,draw=mycolorD0D0E0]  (17.34,41.48) -- (143.77,41.48);
\path [fill=mycolor000000,draw=mycolorD0D0E0]  (17.34,32.86) -- (143.77,32.86);
\path [fill=mycolor000000,draw=mycolorD0D0E0]  (17.34,24.24) -- (143.77,24.24);
\path [fill=mycolor000000,draw=mycolorD0D0E0]  (17.34,15.62) -- (143.77,15.62);
\path [fill=mycolor000000,draw=mycolorD0D0E0]  (17.34,7) -- (143.77,7);
\end{scope}
\begin{scope}
[dash pattern=on 0.5mm off 0.5mm,line width=0.2mm]
\path [fill=mycolor000000,draw=mycolorD0D0E0]  (19.34,5) -- (19.34,60.72);
\path [fill=mycolor000000,draw=mycolorD0D0E0]  (43.83,5) -- (43.83,60.72);
\path [fill=mycolor000000,draw=mycolorD0D0E0]  (68.31,5) -- (68.31,60.72);
\path [fill=mycolor000000,draw=mycolorD0D0E0]  (92.8,5) -- (92.8,60.72);
\path [fill=mycolor000000,draw=mycolorD0D0E0]  (117.28,5) -- (117.28,60.72);
\path [fill=mycolor000000,draw=mycolorD0D0E0]  (141.77,5) -- (141.77,60.72);
\end{scope}
\begin{scope}
\begin{scope}
[line width=0.3mm]
\path [fill=mycolor000000,fill opacity=0,draw=mycolor000000]  (20.56,-27.48) -- (21.18,-27.48) -- (21.79,-27.48) -- (22.4,-27.48) -- (23.01,-27.48) -- (23.63,-27.48) -- (24.24,-27.48) -- (24.85,-27.48) -- (25.46,-27.48) -- (26.07,-27.48) -- (26.69,-27.48) -- (27.3,-27.48) -- (27.91,-27.48) -- (28.52,-27.48) -- (29.13,-27.48) -- (29.75,-27.48) -- (30.36,-27.48) -- (30.97,-27.48) -- (31.58,-27.48) -- (32.2,-27.48) -- (32.81,-27.48) -- (33.42,-27.48) -- (34.03,-27.48) -- (34.64,-27.48) -- (35.26,-27.48) -- (35.87,-27.48) -- (36.48,-27.48) -- (37.09,-27.48) -- (37.7,-27.48) -- (38.32,-27.48) -- (38.93,-27.48) -- (39.54,-27.48) -- (40.15,-27.48) -- (40.77,-27.48) -- (41.38,-27.48) -- (41.99,-27.48) -- (42.6,-27.48) -- (43.21,-27.48) -- (43.83,-27.48) -- (44.44,0.96) -- (45.05,18.79) -- (45.66,28.2) -- (46.27,34.02) -- (46.89,37.97) -- (47.5,40.82) -- (48.11,42.98) -- (48.72,44.67) -- (49.34,46.03) -- (49.95,47.15) -- (50.56,48.08) -- (51.17,48.87) -- (51.78,49.55) -- (52.4,50.14) -- (53.01,50.66) -- (53.62,51.12) -- (54.23,51.52) -- (54.84,51.88) -- (55.46,52.21) -- (56.07,52.5) -- (56.68,52.77) -- (57.29,53.02) -- (57.91,53.24) -- (58.52,53.44) -- (59.13,53.63) -- (59.74,53.81) -- (60.35,53.97) -- (60.97,54.12) -- (61.58,54.26) -- (62.19,54.39) -- (62.8,54.52) -- (63.42,54.63) -- (64.03,54.74) -- (64.64,54.84) -- (65.25,54.93) -- (65.86,55.02) -- (66.48,55.11) -- (67.09,55.19) -- (67.7,55.26) -- (68.31,55.34) -- (68.92,55.4) -- (69.54,55.47) -- (70.15,55.53) -- (70.76,55.59) -- (71.37,55.64) -- (71.99,55.7) -- (72.6,55.75) -- (73.21,55.79) -- (73.82,55.84) -- (74.43,55.88) -- (75.05,55.92) -- (75.66,55.96) -- (76.27,56) -- (76.88,56.03) -- (77.49,56.07) -- (78.11,56.1) -- (78.72,56.13) -- (79.33,56.16) -- (79.94,56.19) -- (80.56,56.21) -- (81.17,56.24) -- (81.78,56.26) -- (82.39,56.28) -- (83,56.3) -- (83.62,56.32) -- (84.23,56.34) -- (84.84,56.36) -- (85.45,56.37) -- (86.06,56.39) -- (86.68,56.4) -- (87.29,56.41) -- (87.9,56.43) -- (88.51,56.44) -- (89.13,56.45) -- (89.74,56.45) -- (90.35,56.46) -- (90.96,56.47) -- (91.57,56.47) -- (92.19,56.47) -- (92.8,56.48) -- (93.41,56.48) -- (94.02,56.48) -- (94.63,56.48) -- (95.25,56.47) -- (95.86,56.47) -- (96.47,56.46) -- (97.08,56.45) -- (97.7,56.44) -- (98.31,56.43) -- (98.92,56.42) -- (99.53,56.4) -- (100.14,56.38) -- (100.76,56.36) -- (101.37,56.34) -- (101.98,56.31) -- (102.59,56.28) -- (103.21,56.24) -- (103.82,56.21) -- (104.43,56.16) -- (105.04,56.11) -- (105.65,56.06) -- (106.27,56) -- (106.88,55.93) -- (107.49,55.85) -- (108.1,55.75) -- (108.71,55.65) -- (109.33,55.53) -- (109.94,55.39) -- (110.55,55.22) -- (111.16,55.03) -- (111.78,54.79) -- (112.39,54.5) -- (113,54.14) -- (113.61,53.67) -- (114.22,53.06) -- (114.84,52.2) -- (115.45,50.94) -- (116.06,48.9) -- (116.67,45.03) -- (117.28,34.89) -- (117.9,-27.48) -- (118.51,-27.48) -- (119.12,-27.48) -- (119.73,-27.48) -- (120.35,-27.48) -- (120.96,-27.48) -- (121.57,-27.48) -- (122.18,-27.48) -- (122.79,-27.48) -- (123.41,-27.48) -- (124.02,-27.48) -- (124.63,-27.48) -- (125.24,-27.48) -- (125.85,-27.48) -- (126.47,-27.48) -- (127.08,-27.48) -- (127.69,-27.48) -- (128.3,-27.48) -- (128.92,-27.48) -- (129.53,-27.48) -- (130.14,-27.48) -- (130.75,-27.48) -- (131.36,-27.48) -- (131.98,-27.48) -- (132.59,-27.48) -- (133.2,-27.48) -- (133.81,-27.48) -- (134.42,-27.48) -- (135.04,-27.48) -- (135.65,-27.48) -- (136.26,-27.48) -- (136.87,-27.48) -- (137.49,-27.48) -- (138.1,-27.48) -- (138.71,-27.48) -- (139.32,-27.48) -- (139.93,-27.48) -- (140.55,-27.48);
\end{scope}
\begin{scope}
[line width=0.3mm]
\path [fill=mycolor000000,fill opacity=0,draw=mycolor00BFFF]  (20.56,-27.48) -- (21.18,-27.48) -- (21.79,-27.48) -- (22.4,-27.48) -- (23.01,-27.48) -- (23.63,-27.48) -- (24.24,-27.48) -- (24.85,-27.48) -- (25.46,-27.48) -- (26.07,-27.48) -- (26.69,-27.48) -- (27.3,-27.48) -- (27.91,-27.48) -- (28.52,-27.48) -- (29.13,-27.48) -- (29.75,-27.48) -- (30.36,-27.48) -- (30.97,-27.48) -- (31.58,-27.48) -- (32.2,-27.48) -- (32.81,-27.48) -- (33.42,-27.48) -- (34.03,-27.48) -- (34.64,-27.48) -- (35.26,-27.48) -- (35.87,-27.48) -- (36.48,-27.48) -- (37.09,-27.48) -- (37.7,-27.48) -- (38.32,-27.48) -- (38.93,-27.48) -- (39.54,-27.48) -- (40.15,-27.48) -- (40.77,-27.48) -- (41.38,-27.48) -- (41.99,-27.48) -- (42.6,16.13) -- (43.21,38.93) -- (43.83,45.82) -- (44.44,49.13) -- (45.05,51.08) -- (45.66,52.37) -- (46.27,53.27) -- (46.89,53.95) -- (47.5,54.47) -- (48.11,54.88) -- (48.72,55.22) -- (49.34,55.5) -- (49.95,55.73) -- (50.56,55.93) -- (51.17,56.1) -- (51.78,56.25) -- (52.4,56.38) -- (53.01,56.5) -- (53.62,56.6) -- (54.23,56.69) -- (54.84,56.78) -- (55.46,56.85) -- (56.07,56.92) -- (56.68,56.98) -- (57.29,57.04) -- (57.91,57.09) -- (58.52,57.14) -- (59.13,57.18) -- (59.74,57.22) -- (60.35,57.26) -- (60.97,57.29) -- (61.58,57.33) -- (62.19,57.36) -- (62.8,57.38) -- (63.42,57.41) -- (64.03,57.43) -- (64.64,57.46) -- (65.25,57.48) -- (65.86,57.5) -- (66.48,57.52) -- (67.09,57.53) -- (67.7,57.55) -- (68.31,57.57) -- (68.92,57.58) -- (69.54,57.59) -- (70.15,57.61) -- (70.76,57.62) -- (71.37,57.63) -- (71.99,57.64) -- (72.6,57.65) -- (73.21,57.66) -- (73.82,57.67) -- (74.43,57.67) -- (75.05,57.68) -- (75.66,57.69) -- (76.27,57.69) -- (76.88,57.7) -- (77.49,57.7) -- (78.11,57.71) -- (78.72,57.71) -- (79.33,57.72) -- (79.94,57.72) -- (80.56,57.72) -- (81.17,57.72) -- (81.78,57.73) -- (82.39,57.73) -- (83,57.73) -- (83.62,57.73) -- (84.23,57.73) -- (84.84,57.73) -- (85.45,57.73) -- (86.06,57.73) -- (86.68,57.73) -- (87.29,57.72) -- (87.9,57.72) -- (88.51,57.72) -- (89.13,57.72) -- (89.74,57.71) -- (90.35,57.71) -- (90.96,57.71) -- (91.57,57.7) -- (92.19,57.7) -- (92.8,57.69) -- (93.41,57.68) -- (94.02,57.68) -- (94.63,57.67) -- (95.25,57.66) -- (95.86,57.66) -- (96.47,57.65) -- (97.08,57.64) -- (97.7,57.63) -- (98.31,57.62) -- (98.92,57.61) -- (99.53,57.6) -- (100.14,57.58) -- (100.76,57.57) -- (101.37,57.56) -- (101.98,57.54) -- (102.59,57.53) -- (103.21,57.51) -- (103.82,57.49) -- (104.43,57.47) -- (105.04,57.45) -- (105.65,57.43) -- (106.27,57.41) -- (106.88,57.39) -- (107.49,57.36) -- (108.1,57.33) -- (108.71,57.31) -- (109.33,57.27) -- (109.94,57.24) -- (110.55,57.21) -- (111.16,57.17) -- (111.78,57.13) -- (112.39,57.09) -- (113,57.04) -- (113.61,56.99) -- (114.22,56.93) -- (114.84,56.87) -- (115.45,56.81) -- (116.06,56.74) -- (116.67,56.66) -- (117.28,56.57) -- (117.9,56.47) -- (118.51,56.37) -- (119.12,56.25) -- (119.73,56.11) -- (120.35,55.96) -- (120.96,55.78) -- (121.57,55.58) -- (122.18,55.34) -- (122.79,55.06) -- (123.41,54.72) -- (124.02,54.3) -- (124.63,53.77) -- (125.24,53.08) -- (125.85,52.14) -- (126.47,50.8) -- (127.08,48.71) -- (127.69,44.98) -- (128.3,36.42) -- (128.92,-4.08) -- (129.53,-27.48) -- (130.14,-27.48) -- (130.75,-27.48) -- (131.36,-27.48) -- (131.98,-27.48) -- (132.59,-27.48) -- (133.2,-27.48) -- (133.81,-27.48) -- (134.42,-27.48) -- (135.04,-27.48) -- (135.65,-27.48) -- (136.26,-27.48) -- (136.87,-27.48) -- (137.49,-27.48) -- (138.1,-27.48) -- (138.71,-27.48) -- (139.32,-27.48) -- (139.93,-27.48) -- (140.55,-27.48);
\path [fill=mycolor000000,fill opacity=0,draw=mycolor00E5B2]  (20.56,-27.48) -- (21.18,-27.48) -- (21.79,-27.48) -- (22.4,-27.48) -- (23.01,-27.48) -- (23.63,-27.48) -- (24.24,-27.48) -- (24.85,-27.48) -- (25.46,-27.48) -- (26.07,-27.48) -- (26.69,-27.48) -- (27.3,-27.48) -- (27.91,-27.48) -- (28.52,-27.48) -- (29.13,-27.48) -- (29.75,-27.48) -- (30.36,-27.48) -- (30.97,-27.48) -- (31.58,-27.48) -- (32.2,-27.48) -- (32.81,-27.48) -- (33.42,-27.48) -- (34.03,-27.48) -- (34.64,-27.48) -- (35.26,-27.48) -- (35.87,-27.48) -- (36.48,-27.48) -- (37.09,-27.48) -- (37.7,-27.48) -- (38.32,-27.48) -- (38.93,-27.48) -- (39.54,-27.48) -- (40.15,-27.48) -- (40.77,-27.48) -- (41.38,-27.48) -- (41.99,-27.48) -- (42.6,-27.48) -- (43.21,-27.48) -- (43.83,-27.48) -- (44.44,-27.48) -- (45.05,-27.48) -- (45.66,-27.48) -- (46.27,-27.48) -- (46.89,-27.48) -- (47.5,-27.48) -- (48.11,-27.48) -- (48.72,-27.48) -- (49.34,-27.48) -- (49.95,-27.48) -- (50.56,-27.48) -- (51.17,-27.48) -- (51.78,-27.48) -- (52.4,-27.48) -- (53.01,-27.48) -- (53.62,-27.48) -- (54.23,-27.48) -- (54.84,-27.48) -- (55.46,-27.48) -- (56.07,-27.48) -- (56.68,-27.48) -- (57.29,-27.48) -- (57.91,-27.48) -- (58.52,-27.48) -- (59.13,-27.48) -- (59.74,-27.48) -- (60.35,-27.48) -- (60.97,-27.48) -- (61.58,-27.48) -- (62.19,-27.48) -- (62.8,-27.48) -- (63.42,-27.48) -- (64.03,-27.48) -- (64.64,-27.48) -- (65.25,-27.48) -- (65.86,-27.48) -- (66.48,-27.48) -- (67.09,-27.48) -- (67.7,-27.48) -- (68.31,-27.48) -- (68.92,-27.48) -- (69.54,-27.48) -- (70.15,-27.48) -- (70.76,-27.48) -- (71.37,-27.48) -- (71.99,-27.48) -- (72.6,-27.48) -- (73.21,-27.48) -- (73.82,-27.48) -- (74.43,-27.48) -- (75.05,-27.48) -- (75.66,-27.48) -- (76.27,-27.48) -- (76.88,-27.48) -- (77.49,-27.48) -- (78.11,-27.48) -- (78.72,-27.48) -- (79.33,-27.48) -- (79.94,-27.48) -- (80.56,-27.48) -- (81.17,-27.48) -- (81.78,-27.48) -- (82.39,-27.48) -- (83,-27.48) -- (83.62,-27.48) -- (84.23,-27.48) -- (84.84,-27.48) -- (85.45,-27.48) -- (86.06,-27.48) -- (86.68,-27.48) -- (87.29,-27.48) -- (87.9,-27.48) -- (88.51,-27.48) -- (89.13,-27.48) -- (89.74,-27.48) -- (90.35,-27.48) -- (90.96,-27.48) -- (91.57,-27.48) -- (92.19,-27.48) -- (92.8,-27.48) -- (93.41,-27.48) -- (94.02,-27.48) -- (94.63,-27.48) -- (95.25,-27.48) -- (95.86,-27.48) -- (96.47,-27.48) -- (97.08,-27.48) -- (97.7,-27.48) -- (98.31,-27.48) -- (98.92,-27.48) -- (99.53,-27.48) -- (100.14,-27.48) -- (100.76,-27.48) -- (101.37,-27.48) -- (101.98,-27.48) -- (102.59,-27.48) -- (103.21,-27.48) -- (103.82,-27.48) -- (104.43,-27.48) -- (105.04,-27.48) -- (105.65,-27.48) -- (106.27,-27.48) -- (106.88,-27.48) -- (107.49,-27.48) -- (108.1,-27.48) -- (108.71,-27.48) -- (109.33,-27.48) -- (109.94,-27.48) -- (110.55,-27.48) -- (111.16,-27.48) -- (111.78,-27.48) -- (112.39,-27.48) -- (113,-27.48) -- (113.61,-27.48) -- (114.22,-27.48) -- (114.84,-27.48) -- (115.45,-27.48) -- (116.06,-27.48) -- (116.67,-27.48) -- (117.28,-27.48) -- (117.9,-27.48) -- (118.51,-27.48) -- (119.12,-27.48) -- (119.73,-27.48) -- (120.35,-27.48) -- (120.96,-27.48) -- (121.57,-27.48) -- (122.18,-27.48) -- (122.79,-27.48) -- (123.41,-27.48) -- (124.02,-27.48) -- (124.63,-27.48) -- (125.24,-27.48) -- (125.85,-27.48) -- (126.47,-27.48) -- (127.08,-27.48) -- (127.69,-27.48) -- (128.3,-27.48) -- (128.92,-27.48) -- (129.53,-27.48) -- (130.14,-27.48) -- (130.75,-27.48) -- (131.36,-27.48) -- (131.98,-27.48) -- (132.59,-27.48) -- (133.2,-27.48) -- (133.81,-27.48) -- (134.42,-27.48) -- (135.04,-27.48) -- (135.65,-27.48) -- (136.26,-27.48) -- (136.87,-27.48) -- (137.49,-27.48) -- (138.1,-27.48) -- (138.71,-27.48) -- (139.32,-27.48) -- (139.93,-27.48) -- (140.55,-27.48);
\path [fill=mycolor000000,fill opacity=0,draw=mycolorD4CA3A]  (20.56,-27.48) -- (21.18,-27.48) -- (21.79,-27.48) -- (22.4,-27.48) -- (23.01,-27.48) -- (23.63,-27.48) -- (24.24,-27.48) -- (24.85,-27.48) -- (25.46,-27.48) -- (26.07,-27.48) -- (26.69,-27.48) -- (27.3,-27.48) -- (27.91,-27.48) -- (28.52,-27.48) -- (29.13,-27.48) -- (29.75,-27.48) -- (30.36,-27.48) -- (30.97,-27.48) -- (31.58,-27.48) -- (32.2,-27.48) -- (32.81,-27.48) -- (33.42,-27.48) -- (34.03,-27.48) -- (34.64,-27.48) -- (35.26,-27.48) -- (35.87,-27.48) -- (36.48,-27.48) -- (37.09,-27.48) -- (37.7,-27.48) -- (38.32,-27.48) -- (38.93,-27.48) -- (39.54,-27.48) -- (40.15,-27.48) -- (40.77,-27.48) -- (41.38,-27.48) -- (41.99,-27.48) -- (42.6,-27.48) -- (43.21,-27.48) -- (43.83,-27.48) -- (44.44,-27.48) -- (45.05,0.4) -- (45.66,25.13) -- (46.27,34.97) -- (46.89,40.25) -- (47.5,43.54) -- (48.11,45.79) -- (48.72,47.41) -- (49.34,48.65) -- (49.95,49.61) -- (50.56,50.38) -- (51.17,51.01) -- (51.78,51.54) -- (52.4,51.99) -- (53.01,52.37) -- (53.62,52.69) -- (54.23,52.98) -- (54.84,53.23) -- (55.46,53.45) -- (56.07,53.65) -- (56.68,53.82) -- (57.29,53.98) -- (57.91,54.12) -- (58.52,54.24) -- (59.13,54.35) -- (59.74,54.46) -- (60.35,54.55) -- (60.97,54.63) -- (61.58,54.71) -- (62.19,54.78) -- (62.8,54.84) -- (63.42,54.9) -- (64.03,54.95) -- (64.64,55) -- (65.25,55.04) -- (65.86,55.08) -- (66.48,55.11) -- (67.09,55.14) -- (67.7,55.17) -- (68.31,55.19) -- (68.92,55.21) -- (69.54,55.22) -- (70.15,55.24) -- (70.76,55.25) -- (71.37,55.25) -- (71.99,55.26) -- (72.6,55.26) -- (73.21,55.26) -- (73.82,55.26) -- (74.43,55.25) -- (75.05,55.24) -- (75.66,55.23) -- (76.27,55.22) -- (76.88,55.21) -- (77.49,55.19) -- (78.11,55.17) -- (78.72,55.14) -- (79.33,55.12) -- (79.94,55.09) -- (80.56,55.05) -- (81.17,55.02) -- (81.78,54.98) -- (82.39,54.94) -- (83,54.89) -- (83.62,54.84) -- (84.23,54.79) -- (84.84,54.73) -- (85.45,54.67) -- (86.06,54.6) -- (86.68,54.53) -- (87.29,54.45) -- (87.9,54.37) -- (88.51,54.28) -- (89.13,54.18) -- (89.74,54.08) -- (90.35,53.97) -- (90.96,53.84) -- (91.57,53.71) -- (92.19,53.57) -- (92.8,53.41) -- (93.41,53.24) -- (94.02,53.05) -- (94.63,52.85) -- (95.25,52.62) -- (95.86,52.37) -- (96.47,52.09) -- (97.08,51.78) -- (97.7,51.43) -- (98.31,51.04) -- (98.92,50.59) -- (99.53,50.07) -- (100.14,49.47) -- (100.76,48.77) -- (101.37,47.93) -- (101.98,46.92) -- (102.59,45.67) -- (103.21,44.1) -- (103.82,42.05) -- (104.43,39.27) -- (105.04,35.29) -- (105.65,29.11) -- (106.27,18.22) -- (106.88,-6.12) -- (107.49,-27.48) -- (108.1,-27.48) -- (108.71,-27.48) -- (109.33,-27.48) -- (109.94,-27.48) -- (110.55,-27.48) -- (111.16,-27.48) -- (111.78,-27.48) -- (112.39,-27.48) -- (113,-27.48) -- (113.61,-27.48) -- (114.22,-27.48) -- (114.84,-27.48) -- (115.45,-27.48) -- (116.06,-27.48) -- (116.67,-27.48) -- (117.28,-27.48) -- (117.9,-27.48) -- (118.51,-27.48) -- (119.12,-27.48) -- (119.73,-27.48) -- (120.35,-27.48) -- (120.96,-27.48) -- (121.57,-27.48) -- (122.18,-27.48) -- (122.79,-27.48) -- (123.41,-27.48) -- (124.02,-27.48) -- (124.63,-27.48) -- (125.24,-27.48) -- (125.85,-27.48) -- (126.47,-27.48) -- (127.08,-27.48) -- (127.69,-27.48) -- (128.3,-27.48) -- (128.92,-27.48) -- (129.53,-27.48) -- (130.14,-27.48) -- (130.75,-27.48) -- (131.36,-27.48) -- (131.98,-27.48) -- (132.59,-27.48) -- (133.2,-27.48) -- (133.81,-27.48) -- (134.42,-27.48) -- (135.04,-27.48) -- (135.65,-27.48) -- (136.26,-27.48) -- (136.87,-27.48) -- (137.49,-27.48) -- (138.1,-27.48) -- (138.71,-27.48) -- (139.32,-27.48) -- (139.93,-27.48) -- (140.55,-27.48);
\path [fill=mycolor000000,fill opacity=0,draw=mycolorFF5EA0]  (20.56,-27.48) -- (21.18,-27.48) -- (21.79,-27.48) -- (22.4,-27.48) -- (23.01,-27.48) -- (23.63,-27.48) -- (24.24,-27.48) -- (24.85,-27.48) -- (25.46,-27.48) -- (26.07,-27.48) -- (26.69,-27.48) -- (27.3,-27.48) -- (27.91,-27.48) -- (28.52,-27.48) -- (29.13,-27.48) -- (29.75,-27.48) -- (30.36,-27.48) -- (30.97,-27.48) -- (31.58,-27.48) -- (32.2,-27.48) -- (32.81,-27.48) -- (33.42,-27.48) -- (34.03,-27.48) -- (34.64,-27.48) -- (35.26,-27.48) -- (35.87,-27.48) -- (36.48,-27.48) -- (37.09,-27.48) -- (37.7,-27.48) -- (38.32,-27.48) -- (38.93,-27.48) -- (39.54,-27.48) -- (40.15,-27.48) -- (40.77,-27.48) -- (41.38,-27.48) -- (41.99,-27.48) -- (42.6,-27.48) -- (43.21,-27.48) -- (43.83,-27.48) -- (44.44,-27.48) -- (45.05,-27.48) -- (45.66,-27.48) -- (46.27,-27.48) -- (46.89,-27.48) -- (47.5,-27.48) -- (48.11,-27.48) -- (48.72,-16.43) -- (49.34,7.05) -- (49.95,18.97) -- (50.56,26.17) -- (51.17,30.97) -- (51.78,34.4) -- (52.4,36.97) -- (53.01,38.95) -- (53.62,40.52) -- (54.23,41.8) -- (54.84,42.85) -- (55.46,43.73) -- (56.07,44.47) -- (56.68,45.1) -- (57.29,45.64) -- (57.91,46.11) -- (58.52,46.51) -- (59.13,46.86) -- (59.74,47.16) -- (60.35,47.42) -- (60.97,47.65) -- (61.58,47.84) -- (62.19,48.01) -- (62.8,48.15) -- (63.42,48.26) -- (64.03,48.35) -- (64.64,48.43) -- (65.25,48.48) -- (65.86,48.51) -- (66.48,48.53) -- (67.09,48.53) -- (67.7,48.51) -- (68.31,48.47) -- (68.92,48.42) -- (69.54,48.35) -- (70.15,48.26) -- (70.76,48.15) -- (71.37,48.03) -- (71.99,47.88) -- (72.6,47.71) -- (73.21,47.52) -- (73.82,47.3) -- (74.43,47.05) -- (75.05,46.78) -- (75.66,46.47) -- (76.27,46.12) -- (76.88,45.72) -- (77.49,45.28) -- (78.11,44.79) -- (78.72,44.23) -- (79.33,43.59) -- (79.94,42.86) -- (80.56,42.02) -- (81.17,41.05) -- (81.78,39.92) -- (82.39,38.59) -- (83,36.99) -- (83.62,35.05) -- (84.23,32.65) -- (84.84,29.61) -- (85.45,25.63) -- (86.06,20.22) -- (86.68,12.45) -- (87.29,0.34) -- (87.9,-21.1) -- (88.51,-27.48) -- (89.13,-27.48) -- (89.74,-27.48) -- (90.35,-27.48) -- (90.96,-27.48) -- (91.57,-27.48) -- (92.19,-27.48) -- (92.8,-27.48) -- (93.41,-27.48) -- (94.02,-27.48) -- (94.63,-27.48) -- (95.25,-27.48) -- (95.86,-27.48) -- (96.47,-27.48) -- (97.08,-27.48) -- (97.7,-27.48) -- (98.31,-27.48) -- (98.92,-27.48) -- (99.53,-27.48) -- (100.14,-27.48) -- (100.76,-27.48) -- (101.37,-27.48) -- (101.98,-27.48) -- (102.59,-27.48) -- (103.21,-27.48) -- (103.82,-27.48) -- (104.43,-27.48) -- (105.04,-27.48) -- (105.65,-27.48) -- (106.27,-27.48) -- (106.88,-27.48) -- (107.49,-27.48) -- (108.1,-27.48) -- (108.71,-27.48) -- (109.33,-27.48) -- (109.94,-27.48) -- (110.55,-27.48) -- (111.16,-27.48) -- (111.78,-27.48) -- (112.39,-27.48) -- (113,-27.48) -- (113.61,-27.48) -- (114.22,-27.48) -- (114.84,-27.48) -- (115.45,-27.48) -- (116.06,-27.48) -- (116.67,-27.48) -- (117.28,-27.48) -- (117.9,-27.48) -- (118.51,-27.48) -- (119.12,-27.48) -- (119.73,-27.48) -- (120.35,-27.48) -- (120.96,-27.48) -- (121.57,-27.48) -- (122.18,-27.48) -- (122.79,-27.48) -- (123.41,-27.48) -- (124.02,-27.48) -- (124.63,-27.48) -- (125.24,-27.48) -- (125.85,-27.48) -- (126.47,-27.48) -- (127.08,-27.48) -- (127.69,-27.48) -- (128.3,-27.48) -- (128.92,-27.48) -- (129.53,-27.48) -- (130.14,-27.48) -- (130.75,-27.48) -- (131.36,-27.48) -- (131.98,-27.48) -- (132.59,-27.48) -- (133.2,-27.48) -- (133.81,-27.48) -- (134.42,-27.48) -- (135.04,-27.48) -- (135.65,-27.48) -- (136.26,-27.48) -- (136.87,-27.48) -- (137.49,-27.48) -- (138.1,-27.48) -- (138.71,-27.48) -- (139.32,-27.48) -- (139.93,-27.48) -- (140.55,-27.48);
\end{scope}
\end{scope}
\end{scope}
\begin{scope}
\draw (16.34,58.72) node [text=mycolor6C606B,rotate around={-0: (-2.51,-25.86)},left,inner sep=0.0]{\fontsize{2.82mm}{3.39mm}\selectfont $\text{0}$};
\draw (16.34,50.1) node [text=mycolor6C606B,rotate around={-0: (-2.51,-17.24)},left,inner sep=0.0]{\fontsize{2.82mm}{3.39mm}\selectfont $\text{20}$};
\draw (16.34,41.48) node [text=mycolor6C606B,rotate around={-0: (-2.51,-8.62)},left,inner sep=0.0]{\fontsize{2.82mm}{3.39mm}\selectfont $\text{40}$};
\draw (16.34,32.86) node [text=mycolor6C606B,rotate around={-0: (-2.51,0)},left,inner sep=0.0]{\fontsize{2.82mm}{3.39mm}\selectfont $\text{60}$};
\draw (16.34,24.24) node [text=mycolor6C606B,rotate around={-0: (-2.51,8.62)},left,inner sep=0.0]{\fontsize{2.82mm}{3.39mm}\selectfont $\text{80}$};
\draw (16.34,15.62) node [text=mycolor6C606B,rotate around={-0: (-2.51,17.24)},left,inner sep=0.0]{\fontsize{2.82mm}{3.39mm}\selectfont $\text{100}$};
\draw (16.34,7) node [text=mycolor6C606B,rotate around={-0: (-2.51,25.86)},left,inner sep=0.0]{\fontsize{2.82mm}{3.39mm}\selectfont $\text{120}$};
\end{scope}
\begin{scope}
\draw (8.32,32.86) node [text=mycolor564A55,draw=mycolor000000,draw opacity=0,rotate around={-0: (-0.66,0)},left,inner sep=0.0]{\fontsize{3.88mm}{4.66mm}\selectfont $\text{t}$};
\end{scope}
\end{scope}
\end{tikzpicture}

%% file: lambda1Et2_bis.tex
\begin{tikzpicture}[x=1mm,y=-1mm]
\definecolor{mycolor00BFFF}{rgb}{0,0.75,1}
\definecolor{mycolor000000}{rgb}{0,0,0}
\definecolor{mycolor564A55}{rgb}{0.34,0.29,0.33}
\definecolor{mycolorD4CA3A}{rgb}{0.83,0.79,0.23}
\definecolor{mycolor6C606B}{rgb}{0.42,0.38,0.42}
\definecolor{mycolor362A35}{rgb}{0.21,0.16,0.21}
\definecolor{mycolorD0D0E0}{rgb}{0.82,0.82,0.88}
\definecolor{mycolor000000}{rgb}{0,0,0}
\definecolor{mycolor4C404B}{rgb}{0.3,0.25,0.29}
\begin{scope}
\begin{scope}
\draw (80.56,68.39) node [text=mycolor564A55,draw=mycolor000000,draw opacity=0,rotate around={-0: (0,1.81)},inner sep=0.0]{\fontsize{3.88mm}{4.66mm}\selectfont $\text{s}$};
\end{scope}
\begin{scope}
\draw (19.34,61.72) node [text=mycolor6C606B,rotate around={-0: (61.22,1.34)},inner sep=0.0]{\fontsize{2.82mm}{3.39mm}\selectfont $\text{0.0}$};
\draw (43.83,61.72) node [text=mycolor6C606B,rotate around={-0: (36.73,1.34)},inner sep=0.0]{\fontsize{2.82mm}{3.39mm}\selectfont $\text{0.2}$};
\draw (68.31,61.72) node [text=mycolor6C606B,rotate around={-0: (12.24,1.34)},inner sep=0.0]{\fontsize{2.82mm}{3.39mm}\selectfont $\text{0.4}$};
\draw (92.8,61.72) node [text=mycolor6C606B,rotate around={-0: (-12.24,1.34)},inner sep=0.0]{\fontsize{2.82mm}{3.39mm}\selectfont $\text{0.6}$};
\draw (117.28,61.72) node [text=mycolor6C606B,rotate around={-0: (-36.73,1.34)},inner sep=0.0]{\fontsize{2.82mm}{3.39mm}\selectfont $\text{0.8}$};
\draw (141.77,61.72) node [text=mycolor6C606B,rotate around={-0: (-61.22,1.34)},inner sep=0.0]{\fontsize{2.82mm}{3.39mm}\selectfont $\text{1.0}$};
\end{scope}
\begin{scope}
\begin{scope}
\draw (147.58,32.86) node [text=mycolor4C404B,rotate around={-0: (-0.06,0)},right,inner sep=0.0]{\fontsize{2.82mm}{3.39mm}\selectfont $\text{6.2}$};
\draw (147.58,36.48) node [text=mycolor4C404B,rotate around={-0: (-0.06,-3.63)},right,inner sep=0.0]{\fontsize{2.82mm}{3.39mm}\selectfont $\text{6.7}$};
\end{scope}
\begin{scope}
\path [fill=mycolor00BFFF,draw=mycolor000000,draw opacity=0] (144.77,31.95) rectangle +(1.81,1.81);
\path [fill=mycolorD4CA3A,draw=mycolor000000,draw opacity=0] (144.77,35.58) rectangle +(1.81,1.81);
\end{scope}
\begin{scope}
\draw (144.77,29.04) node [text=mycolor362A35,draw=mycolor000000,draw opacity=0,rotate around={-0: (5.61,0.19)},right,inner sep=0.0]{\fontsize{3.88mm}{4.66mm}\selectfont $\text{$\rho$}$};
\end{scope}
\end{scope}
\begin{scope}
\clip  (17.34,5) -- (143.77,5) -- (143.77,60.72) -- (17.34,60.72);
\begin{scope}
\clip  (17.34,5) -- (143.77,5) -- (143.77,60.72) -- (17.34,60.72);
\path [fill=mycolor000000,fill opacity=0,draw=mycolor000000,draw opacity=0] (17.34,5) rectangle +(126.43,55.72);
\end{scope}
\begin{scope}
[dash pattern=on 0.5mm off 0.5mm,line width=0.2mm]
\path [fill=mycolor000000,draw=mycolorD0D0E0]  (17.34,58.72) -- (143.77,58.72);
\path [fill=mycolor000000,draw=mycolorD0D0E0]  (17.34,50.1) -- (143.77,50.1);
\path [fill=mycolor000000,draw=mycolorD0D0E0]  (17.34,41.48) -- (143.77,41.48);
\path [fill=mycolor000000,draw=mycolorD0D0E0]  (17.34,32.86) -- (143.77,32.86);
\path [fill=mycolor000000,draw=mycolorD0D0E0]  (17.34,24.24) -- (143.77,24.24);
\path [fill=mycolor000000,draw=mycolorD0D0E0]  (17.34,15.62) -- (143.77,15.62);
\path [fill=mycolor000000,draw=mycolorD0D0E0]  (17.34,7) -- (143.77,7);
\end{scope}
\begin{scope}
[dash pattern=on 0.5mm off 0.5mm,line width=0.2mm]
\path [fill=mycolor000000,draw=mycolorD0D0E0]  (19.34,5) -- (19.34,60.72);
\path [fill=mycolor000000,draw=mycolorD0D0E0]  (43.83,5) -- (43.83,60.72);
\path [fill=mycolor000000,draw=mycolorD0D0E0]  (68.31,5) -- (68.31,60.72);
\path [fill=mycolor000000,draw=mycolorD0D0E0]  (92.8,5) -- (92.8,60.72);
\path [fill=mycolor000000,draw=mycolorD0D0E0]  (117.28,5) -- (117.28,60.72);
\path [fill=mycolor000000,draw=mycolorD0D0E0]  (141.77,5) -- (141.77,60.72);
\end{scope}
\begin{scope}
\begin{scope}
[line width=0.3mm]
\path [fill=mycolor000000,fill opacity=0,draw=mycolor000000]  (20.56,-27.48) -- (21.18,-27.48) -- (21.79,-27.48) -- (22.4,-27.48) -- (23.01,-27.48) -- (23.63,-27.48) -- (24.24,-27.48) -- (24.85,-27.48) -- (25.46,-27.48) -- (26.07,-27.48) -- (26.69,-27.48) -- (27.3,-27.48) -- (27.91,-27.48) -- (28.52,-27.48) -- (29.13,-27.48) -- (29.75,-27.48) -- (30.36,-27.48) -- (30.97,-27.48) -- (31.58,-27.48) -- (32.2,-27.48) -- (32.81,-27.48) -- (33.42,-27.48) -- (34.03,-27.48) -- (34.64,-27.48) -- (35.26,-27.48) -- (35.87,-27.48) -- (36.48,-27.48) -- (37.09,-27.48) -- (37.7,-27.48) -- (38.32,-27.48) -- (38.93,-27.48) -- (39.54,-27.48) -- (40.15,-27.48) -- (40.77,-27.48) -- (41.38,-27.48) -- (41.99,-27.48) -- (42.6,-27.48) -- (43.21,-27.48) -- (43.83,-27.48) -- (44.44,-27.48) -- (45.05,-27.48) -- (45.66,-27.48) -- (46.27,-27.48) -- (46.89,-27.48) -- (47.5,-27.48) -- (48.11,-27.48) -- (48.72,-27.48) -- (49.34,-27.48) -- (49.95,-27.48) -- (50.56,-27.48) -- (51.17,-27.48) -- (51.78,-27.48) -- (52.4,-27.48) -- (53.01,-27.48) -- (53.62,-27.48) -- (54.23,-27.48) -- (54.84,-27.48) -- (55.46,-27.48) -- (56.07,-27.48) -- (56.68,-19.03) -- (57.29,-4.14) -- (57.91,5.62) -- (58.52,12.48) -- (59.13,17.55) -- (59.74,21.44) -- (60.35,24.49) -- (60.97,26.94) -- (61.58,28.93) -- (62.19,30.57) -- (62.8,31.93) -- (63.42,33.07) -- (64.03,34.01) -- (64.64,34.79) -- (65.25,35.43) -- (65.86,35.95) -- (66.48,36.36) -- (67.09,36.66) -- (67.7,36.87) -- (68.31,36.98) -- (68.92,37) -- (69.54,36.92) -- (70.15,36.74) -- (70.76,36.45) -- (71.37,36.04) -- (71.99,35.49) -- (72.6,34.77) -- (73.21,33.87) -- (73.82,32.72) -- (74.43,31.28) -- (75.05,29.43) -- (75.66,27.05) -- (76.27,23.91) -- (76.88,19.63) -- (77.49,13.54) -- (78.11,4.3) -- (78.72,-11.23) -- (79.33,-27.48) -- (79.94,-27.48) -- (80.56,-27.48) -- (81.17,-27.48) -- (81.78,-27.48) -- (82.39,-27.48) -- (83,-27.48) -- (83.62,-27.48) -- (84.23,-27.48) -- (84.84,-27.48) -- (85.45,-27.48) -- (86.06,-27.48) -- (86.68,-27.48) -- (87.29,-27.48) -- (87.9,-27.48) -- (88.51,-27.48) -- (89.13,-27.48) -- (89.74,-27.48) -- (90.35,-27.48) -- (90.96,-27.48) -- (91.57,-27.48) -- (92.19,-27.48) -- (92.8,-27.48) -- (93.41,-27.48) -- (94.02,-27.48) -- (94.63,-27.48) -- (95.25,-27.48) -- (95.86,-27.48) -- (96.47,-27.48) -- (97.08,-27.48) -- (97.7,-27.48) -- (98.31,-27.48) -- (98.92,-27.48) -- (99.53,-27.48) -- (100.14,-27.48) -- (100.76,-27.48) -- (101.37,-27.48) -- (101.98,-27.48) -- (102.59,-27.48) -- (103.21,-27.48) -- (103.82,-27.48) -- (104.43,-27.48) -- (105.04,-27.48) -- (105.65,-27.48) -- (106.27,-27.48) -- (106.88,-27.48) -- (107.49,-27.48) -- (108.1,-27.48) -- (108.71,-27.48) -- (109.33,-27.48) -- (109.94,-27.48) -- (110.55,-27.48) -- (111.16,-27.48) -- (111.78,-27.48) -- (112.39,-27.48) -- (113,-27.48) -- (113.61,-27.48) -- (114.22,-27.48) -- (114.84,-27.48) -- (115.45,-27.48) -- (116.06,-27.48) -- (116.67,-27.48) -- (117.28,-27.48) -- (117.9,-27.48) -- (118.51,-27.48) -- (119.12,-27.48) -- (119.73,-27.48) -- (120.35,-27.48) -- (120.96,-27.48) -- (121.57,-27.48) -- (122.18,-27.48) -- (122.79,-27.48) -- (123.41,-27.48) -- (124.02,-27.48) -- (124.63,-27.48) -- (125.24,-27.48) -- (125.85,-27.48) -- (126.47,-27.48) -- (127.08,-27.48) -- (127.69,-27.48) -- (128.3,-27.48) -- (128.92,-27.48) -- (129.53,-27.48) -- (130.14,-27.48) -- (130.75,-27.48) -- (131.36,-27.48) -- (131.98,-27.48) -- (132.59,-27.48) -- (133.2,-27.48) -- (133.81,-27.48) -- (134.42,-27.48) -- (135.04,-27.48) -- (135.65,-27.48) -- (136.26,-27.48) -- (136.87,-27.48) -- (137.49,-27.48) -- (138.1,-27.48) -- (138.71,-27.48) -- (139.32,-27.48) -- (139.93,-27.48) -- (140.55,-27.48);
\end{scope}
\begin{scope}
[line width=0.3mm]
\path [fill=mycolor000000,fill opacity=0,draw=mycolor00BFFF]  (20.56,-27.48) -- (21.18,-27.48) -- (21.79,-27.48) -- (22.4,-27.48) -- (23.01,-27.48) -- (23.63,-27.48) -- (24.24,-27.48) -- (24.85,-27.48) -- (25.46,-27.48) -- (26.07,-27.48) -- (26.69,-27.48) -- (27.3,-27.48) -- (27.91,-27.48) -- (28.52,-27.48) -- (29.13,-27.48) -- (29.75,-27.48) -- (30.36,-27.48) -- (30.97,-27.48) -- (31.58,-27.48) -- (32.2,-27.48) -- (32.81,-27.48) -- (33.42,-27.48) -- (34.03,-27.48) -- (34.64,-27.48) -- (35.26,-27.48) -- (35.87,-27.48) -- (36.48,-22.16) -- (37.09,11.07) -- (37.7,24.42) -- (38.32,31.59) -- (38.93,36.04) -- (39.54,39.05) -- (40.15,41.21) -- (40.77,42.83) -- (41.38,44.07) -- (41.99,45.04) -- (42.6,45.81) -- (43.21,46.44) -- (43.83,46.94) -- (44.44,47.34) -- (45.05,47.67) -- (45.66,47.92) -- (46.27,48.11) -- (46.89,48.25) -- (47.5,48.34) -- (48.11,48.39) -- (48.72,48.39) -- (49.34,48.35) -- (49.95,48.27) -- (50.56,48.14) -- (51.17,47.96) -- (51.78,47.73) -- (52.4,47.44) -- (53.01,47.09) -- (53.62,46.66) -- (54.23,46.14) -- (54.84,45.51) -- (55.46,44.75) -- (56.07,43.81) -- (56.68,42.65) -- (57.29,41.19) -- (57.91,39.29) -- (58.52,36.78) -- (59.13,33.28) -- (59.74,28.13) -- (60.35,19.82) -- (60.97,4.27) -- (61.58,-27.48) -- (62.19,-27.48) -- (62.8,-27.48) -- (63.42,-27.48) -- (64.03,-27.48) -- (64.64,-27.48) -- (65.25,-27.48) -- (65.86,-27.48) -- (66.48,-27.48) -- (67.09,-27.48) -- (67.7,-27.48) -- (68.31,-27.48) -- (68.92,-27.48) -- (69.54,-27.48) -- (70.15,-27.48) -- (70.76,-27.48) -- (71.37,-27.48) -- (71.99,-27.48) -- (72.6,-27.48) -- (73.21,-27.48) -- (73.82,-27.48) -- (74.43,-27.48) -- (75.05,-27.48) -- (75.66,-27.48) -- (76.27,-27.48) -- (76.88,-27.48) -- (77.49,-27.48) -- (78.11,-27.48) -- (78.72,-27.48) -- (79.33,-27.48) -- (79.94,-27.48) -- (80.56,-27.48) -- (81.17,-27.48) -- (81.78,-27.48) -- (82.39,-27.48) -- (83,-27.48) -- (83.62,-27.48) -- (84.23,-27.48) -- (84.84,-27.48) -- (85.45,-27.48) -- (86.06,-27.48) -- (86.68,-27.48) -- (87.29,-27.48) -- (87.9,-27.48) -- (88.51,-27.48) -- (89.13,-27.48) -- (89.74,-27.48) -- (90.35,-27.48) -- (90.96,-27.48) -- (91.57,-27.48) -- (92.19,-27.48) -- (92.8,-27.48) -- (93.41,-27.48) -- (94.02,-27.48) -- (94.63,-27.48) -- (95.25,-27.48) -- (95.86,-27.48) -- (96.47,-27.48) -- (97.08,-27.48) -- (97.7,-27.48) -- (98.31,-27.48) -- (98.92,-27.48) -- (99.53,-27.48) -- (100.14,-27.48) -- (100.76,-27.48) -- (101.37,-27.48) -- (101.98,-27.48) -- (102.59,-27.48) -- (103.21,-27.48) -- (103.82,-27.48) -- (104.43,-27.48) -- (105.04,-27.48) -- (105.65,-27.48) -- (106.27,-27.48) -- (106.88,-27.48) -- (107.49,-27.48) -- (108.1,-27.48) -- (108.71,-27.48) -- (109.33,-27.48) -- (109.94,-27.48) -- (110.55,-27.48) -- (111.16,-27.48) -- (111.78,-27.48) -- (112.39,-27.48) -- (113,-27.48) -- (113.61,-27.48) -- (114.22,-27.48) -- (114.84,-27.48) -- (115.45,-27.48) -- (116.06,-27.48) -- (116.67,-27.48) -- (117.28,-27.48) -- (117.9,-27.48) -- (118.51,-27.48) -- (119.12,-27.48) -- (119.73,-27.48) -- (120.35,-27.48) -- (120.96,-27.48) -- (121.57,-27.48) -- (122.18,-27.48) -- (122.79,-27.48) -- (123.41,-27.48) -- (124.02,-27.48) -- (124.63,-27.48) -- (125.24,-27.48) -- (125.85,-27.48) -- (126.47,-27.48) -- (127.08,-27.48) -- (127.69,-27.48) -- (128.3,-27.48) -- (128.92,-27.48) -- (129.53,-27.48) -- (130.14,-27.48) -- (130.75,-27.48) -- (131.36,-27.48) -- (131.98,-27.48) -- (132.59,-27.48) -- (133.2,-27.48) -- (133.81,-27.48) -- (134.42,-27.48) -- (135.04,-27.48) -- (135.65,-27.48) -- (136.26,-27.48) -- (136.87,-27.48) -- (137.49,-27.48) -- (138.1,-27.48) -- (138.71,-27.48) -- (139.32,-27.48) -- (139.93,-27.48) -- (140.55,-27.48);
\path [fill=mycolor000000,fill opacity=0,draw=mycolorD4CA3A]  (20.56,-27.48) -- (21.18,-27.48) -- (21.79,-27.48) -- (22.4,-27.48) -- (23.01,-27.48) -- (23.63,-27.48) -- (24.24,-27.48) -- (24.85,-27.48) -- (25.46,-27.48) -- (26.07,-27.48) -- (26.69,-27.48) -- (27.3,-27.48) -- (27.91,-27.48) -- (28.52,-27.48) -- (29.13,-27.48) -- (29.75,-27.48) -- (30.36,-27.48) -- (30.97,-27.48) -- (31.58,-27.48) -- (32.2,-27.48) -- (32.81,-27.48) -- (33.42,-27.48) -- (34.03,-27.48) -- (34.64,-27.48) -- (35.26,-27.48) -- (35.87,-27.48) -- (36.48,-27.48) -- (37.09,-27.48) -- (37.7,-27.48) -- (38.32,-27.48) -- (38.93,-27.48) -- (39.54,-27.48) -- (40.15,-19.64) -- (40.77,-1.93) -- (41.38,8.17) -- (41.99,14.55) -- (42.6,18.84) -- (43.21,21.79) -- (43.83,23.84) -- (44.44,25.21) -- (45.05,26.04) -- (45.66,26.42) -- (46.27,26.37) -- (46.89,25.9) -- (47.5,24.97) -- (48.11,23.52) -- (48.72,21.42) -- (49.34,18.43) -- (49.95,14.2) -- (50.56,8.03) -- (51.17,-1.44) -- (51.78,-17.33) -- (52.4,-27.48) -- (53.01,-27.48) -- (53.62,-27.48) -- (54.23,-27.48) -- (54.84,-27.48) -- (55.46,-27.48) -- (56.07,-27.48) -- (56.68,-27.48) -- (57.29,-27.48) -- (57.91,-27.48) -- (58.52,-27.48) -- (59.13,-27.48) -- (59.74,-27.48) -- (60.35,-27.48) -- (60.97,-27.48) -- (61.58,-27.48) -- (62.19,-27.48) -- (62.8,-27.48) -- (63.42,-27.48) -- (64.03,-27.48) -- (64.64,-27.48) -- (65.25,-27.48) -- (65.86,-27.48) -- (66.48,-27.48) -- (67.09,-27.48) -- (67.7,-27.48) -- (68.31,-27.48) -- (68.92,-27.48) -- (69.54,-27.48) -- (70.15,-27.48) -- (70.76,-27.48) -- (71.37,-27.48) -- (71.99,-27.48) -- (72.6,-27.48) -- (73.21,-27.48) -- (73.82,-27.48) -- (74.43,-27.48) -- (75.05,-27.48) -- (75.66,-27.48) -- (76.27,-27.48) -- (76.88,-27.48) -- (77.49,-27.48) -- (78.11,-27.48) -- (78.72,-27.48) -- (79.33,-27.48) -- (79.94,-27.48) -- (80.56,-27.48) -- (81.17,-27.48) -- (81.78,-27.48) -- (82.39,-27.48) -- (83,-27.48) -- (83.62,-27.48) -- (84.23,-27.48) -- (84.84,-27.48) -- (85.45,-27.48) -- (86.06,-27.48) -- (86.68,-27.48) -- (87.29,-27.48) -- (87.9,-27.48) -- (88.51,-27.48) -- (89.13,-27.48) -- (89.74,-27.48) -- (90.35,-27.48) -- (90.96,-27.48) -- (91.57,-27.48) -- (92.19,-27.48) -- (92.8,-27.48) -- (93.41,-27.48) -- (94.02,-27.48) -- (94.63,-27.48) -- (95.25,-27.48) -- (95.86,-27.48) -- (96.47,-27.48) -- (97.08,-27.48) -- (97.7,-27.48) -- (98.31,-27.48) -- (98.92,-27.48) -- (99.53,-27.48) -- (100.14,-27.48) -- (100.76,-27.48) -- (101.37,-27.48) -- (101.98,-27.48) -- (102.59,-27.48) -- (103.21,-27.48) -- (103.82,-27.48) -- (104.43,-27.48) -- (105.04,-27.48) -- (105.65,-27.48) -- (106.27,-27.48) -- (106.88,-27.48) -- (107.49,-27.48) -- (108.1,-27.48) -- (108.71,-27.48) -- (109.33,-27.48) -- (109.94,-27.48) -- (110.55,-27.48) -- (111.16,-27.48) -- (111.78,-27.48) -- (112.39,-27.48) -- (113,-27.48) -- (113.61,-27.48) -- (114.22,-27.48) -- (114.84,-27.48) -- (115.45,-27.48) -- (116.06,-27.48) -- (116.67,-27.48) -- (117.28,-27.48) -- (117.9,-27.48) -- (118.51,-27.48) -- (119.12,-27.48) -- (119.73,-27.48) -- (120.35,-27.48) -- (120.96,-27.48) -- (121.57,-27.48) -- (122.18,-27.48) -- (122.79,-27.48) -- (123.41,-27.48) -- (124.02,-27.48) -- (124.63,-27.48) -- (125.24,-27.48) -- (125.85,-27.48) -- (126.47,-27.48) -- (127.08,-27.48) -- (127.69,-27.48) -- (128.3,-27.48) -- (128.92,-27.48) -- (129.53,-27.48) -- (130.14,-27.48) -- (130.75,-27.48) -- (131.36,-27.48) -- (131.98,-27.48) -- (132.59,-27.48) -- (133.2,-27.48) -- (133.81,-27.48) -- (134.42,-27.48) -- (135.04,-27.48) -- (135.65,-27.48) -- (136.26,-27.48) -- (136.87,-27.48) -- (137.49,-27.48) -- (138.1,-27.48) -- (138.71,-27.48) -- (139.32,-27.48) -- (139.93,-27.48) -- (140.55,-27.48);
\end{scope}
\end{scope}
\end{scope}
\begin{scope}
\draw (16.34,58.72) node [text=mycolor6C606B,rotate around={-0: (-2.51,-25.86)},left,inner sep=0.0]{\fontsize{2.82mm}{3.39mm}\selectfont $\text{0}$};
\draw (16.34,50.1) node [text=mycolor6C606B,rotate around={-0: (-2.51,-17.24)},left,inner sep=0.0]{\fontsize{2.82mm}{3.39mm}\selectfont $\text{20}$};
\draw (16.34,41.48) node [text=mycolor6C606B,rotate around={-0: (-2.51,-8.62)},left,inner sep=0.0]{\fontsize{2.82mm}{3.39mm}\selectfont $\text{40}$};
\draw (16.34,32.86) node [text=mycolor6C606B,rotate around={-0: (-2.51,0)},left,inner sep=0.0]{\fontsize{2.82mm}{3.39mm}\selectfont $\text{60}$};
\draw (16.34,24.24) node [text=mycolor6C606B,rotate around={-0: (-2.51,8.62)},left,inner sep=0.0]{\fontsize{2.82mm}{3.39mm}\selectfont $\text{80}$};
\draw (16.34,15.62) node [text=mycolor6C606B,rotate around={-0: (-2.51,17.24)},left,inner sep=0.0]{\fontsize{2.82mm}{3.39mm}\selectfont $\text{100}$};
\draw (16.34,7) node [text=mycolor6C606B,rotate around={-0: (-2.51,25.86)},left,inner sep=0.0]{\fontsize{2.82mm}{3.39mm}\selectfont $\text{120}$};
\end{scope}
\begin{scope}
\draw (8.32,32.86) node [text=mycolor564A55,draw=mycolor000000,draw opacity=0,rotate around={-0: (-0.66,0)},left,inner sep=0.0]{\fontsize{3.88mm}{4.66mm}\selectfont $\text{t}$};
\end{scope}
\end{scope}
\end{tikzpicture}

%% file: dessins.tex
\newcommand{\drawAxes}{%
  \begin{scope}[help lines,->]
    \draw (0,0) -- (5.3,0);
    \draw (0,0) -- (0,4);
  \end{scope}
}
\newcommand{\setNodeCoordinates}[4]{%
  \begin{scope}[label distance=5pt]
  \coordinate (\ns-O) at (0,0)    {};
  \coordinate[label=below:$1/a$] (\ns-A) at (1/#1,0) {};
  \coordinate[label=below:$1/c$] (\ns-C) at (1/#3,0) {};
  \coordinate[label=left:$1/b$] (\ns-B) at (0,1/#2) {};
  \coordinate[label=left:$1/d$] (\ns-D) at (0,1/#4) {};
  \end{scope}
}
\newcommand{\setIntersectionCoordinates}{%
  \path[name path=horizontal] (\ns-A)--(\ns-B);
  \path[name path=vertical]   (\ns-C)--(\ns-D);
  \path[name intersections={of=horizontal and vertical,by={\ns-I}}] ;
}
\newcommand{\drawIsoclines}{%
   \draw[ 
     postaction={decorate},
     decoration={
       markings,
       mark= between positions 0.1 and 0.9 step 0.2 with {%
	  % les flèches ne doivent pas tourner avec le chemin
	 \pgftransformresetnontranslations
	 \draw[<->,red!50] (0pt,5pt) -- (0pt,-5pt);
       }}]
    (\ns-A)--(\ns-B);
   \draw[ 
     postaction={decorate},
     decoration={
       markings,
       mark= between positions 0.1 and 0.9 step 0.2 with {%
	  % les flèches ne doivent pas tourner avec le chemin
	 \pgftransformresetnontranslations
	 \draw[<->,blue!50] (-5pt,0pt) -- (5pt,0pt);
       }}]
  (\ns-C)--(\ns-D);
}
\newcommand{\stablePoint}[1]{%
  \fill (\ns-#1) circle (2pt);
}
\newcommand{\unstablePoint}[1]{%
  \draw[fill=white] (\ns-#1) circle (2pt);
}
\newcommand{\hyperbolicPoint}[1]{%
  \draw[thick] (\ns-#1) +(2pt,2pt) -- +(-2pt,-2pt);
  \draw[thick] (\ns-#1) +(-2pt,2pt) -- +(2pt,-2pt);
}

\newcommand{\legend}[1]{%
  \node[align=center] at (5,3) [%
    rectangle,rounded corners,
    draw,fill=white]%
    {#1};
}

\newcommand{\ns}{}
%\begin{tikzpicture}[scale=3]
%  \setNodeCoordinates{1}{0.66}{2}{1.33}
%  \begin{scope}[blue]
%  \drawIsoclines
%  \stablePoint{A}
%  \hyperbolicPoint{D}
%  \unstablePoint{O}
%  \end{scope}
%
%  \begin{scope}[red]
%  \renewcommand{\ns}{Bla}
%  \setNodeCoordinates{3}{3}{4}{4}
%  \drawIsoclines
%  \stablePoint{A}
%  \hyperbolicPoint{D}
%  \unstablePoint{O}
%  \end{scope}
%\end{tikzpicture}
%
%\end{document}

\begin{tikzpicture}%[scale=1.2]
  \drawAxes
  \setNodeCoordinates{0.2}{0.4}{1}{0.8}
  \drawIsoclines
  \draw[help lines] plot file{case_1.txt};
  \stablePoint{A}
  \hyperbolicPoint{D}
  \unstablePoint{O}
  \legend{$a<c$, $b<d$\\$(s\in I^c\cap J^c$)}

 % \renewcommand{\ns}{Bla}
 % \setNodeCoordinates{0.5}{0.8}{1}{0.4}
 % \drawIsoclines

\end{tikzpicture}
\hfill
\begin{tikzpicture}%[scale=1.2]
  \drawAxes
  \setNodeCoordinates{0.2}{0.8}{1}{0.4}
  \setIntersectionCoordinates
  \draw[help lines] plot file{case_2.txt};
  \drawIsoclines
  \stablePoint{A}
  \stablePoint{D}
  \unstablePoint{O}
  \hyperbolicPoint{I}
  \legend{$a<c$, $b>d$\\$(s\in I^c\cap J$)}
\end{tikzpicture}

\bigskip

\begin{tikzpicture}%[scale=1.2]
  \drawAxes
  \setNodeCoordinates{1}{0.4}{0.2}{0.8}
  \setIntersectionCoordinates
  \draw[help lines] plot file{case_3.txt};
  \drawIsoclines
  \hyperbolicPoint{A}
  \hyperbolicPoint{D}
  \unstablePoint{O}
  \stablePoint{I}
  \legend{$a>c$, $b<d$\\$(s\in I\cap J^c$)}
\end{tikzpicture}
\hfill
\begin{tikzpicture}%[scale=1.2]
  \drawAxes
  \setNodeCoordinates{1}{0.8}{0.2}{0.4}
  \draw[help lines] plot file{case_4.txt};
  \drawIsoclines
  \hyperbolicPoint{A}
  \stablePoint{D}
  \unstablePoint{O}
  \legend{$a>c$, $b>d$\\$(s\in I\cap J$)}
\end{tikzpicture}

%% file: proof_support.tex
\newcommand{\ns}{}
\newcommand{\stablePoint}[1]{%
  \fill (\ns-#1) circle (2pt);
}
\newcommand{\unstablePoint}[1]{%
  \draw[fill=white] (\ns-#1) circle (2pt);
}
\newcommand{\hyperbolicPoint}[1]{%
  \draw[thick] (\ns-#1) +(2pt,2pt) -- +(-2pt,-2pt);
  \draw[thick] (\ns-#1) +(-2pt,2pt) -- +(2pt,-2pt);
}
\newcommand{\drawAxes}{%
  \begin{scope}[help lines,->]
    \draw (0,0) -- (5.3,0);
    \draw (0,0) -- (0,4);
  \end{scope}
}
\newcommand{\setNodeCoordinates}[5]{%
  % 4 coefficients du champ, l'indice à mettre sur les labels
  \begin{scope}[label distance=5pt]
  \coordinate (\ns-O) at (0,0)    {};
  \coordinate[label=below:$A_{#5}$] (\ns-A) at (1/#1,0) {};
  \coordinate[label=below:$C_{#5}$] (\ns-C) at (1/#3,0) {};
  \coordinate[label=left:$B_{#5}$] (\ns-B) at (0,1/#2) {};
  \coordinate[label=left:$D_{#5}$] (\ns-D) at (0,1/#4) {};
  \end{scope}
}
\newcommand{\setIntersectionCoordinates}{%
  \path[name path=horizontal] (\ns-A)--(\ns-B);
  \path[name path=vertical]   (\ns-C)--(\ns-D);
  \path[name intersections={of=horizontal and vertical,by={\ns-I}}] ;
}
\newcommand{\drawIsoclines}{
  \begin{colormixin}{30!white} % garde 75% de la couleur actuelle
   \draw
    (\ns-A)--(\ns-B)
    (\ns-C)--(\ns-D);
  \end{colormixin}
}
\tikzset{
  set arrow inside/.code={\pgfqkeys{/tikz/arrow inside}{#1}},
  set arrow inside={end/.initial=>, opt/.initial=},
  /pgf/decoration/Mark/.style={
    mark/.expanded=at position #1 with
    {
        \noexpand\arrow[\pgfkeysvalueof{/tikz/arrow inside/opt}]{\pgfkeysvalueof{/tikz/arrow inside/end}}
    }
  },
  arrow inside/.style 2 args={
    set arrow inside={#1},
    postaction={
      decorate,decoration={
        markings,Mark/.list={#2}
      }
    }
  },
}

\begin{tikzpicture}[x=10cm,y=12cm]
  \clip (-0.1,-0.1) rectangle (1.2,0.5);
  \begin{scope}[help lines,->]
    \draw (0,0) -- (2,0);
    \draw (0,0) -- (0,0.4);
  \end{scope}
  % F_0
  \setNodeCoordinates{1}{1}{2}{2.3}{0}
  \begin{scope}[red]
    \draw[thick] plot[smooth] file{mnf_0.txt}%
      [arrow inside={end=To[],opt={scale=1}}{0.25,0.5,0.75}];
    \drawIsoclines
    \hyperbolicPoint{D}
    \stablePoint{A}
  \end{scope}
  % F_1
  \setNodeCoordinates{3}{3.2}{4}{3.9}{1}
  \begin{scope}[blue]
    \draw[thick] plot[smooth] file{mnf_1.txt}%
      [arrow inside={end=To[],opt={scale=1}}{0.25,0.5,0.75}];
    \drawIsoclines
    \hyperbolicPoint{D}
    \stablePoint{A}
  \end{scope}
  % F_S
  \setNodeCoordinates{2.6}{2.7}{2.3}{2.5}{s}
  \draw[thick, red!50!blue] plot[smooth] file{mnf_s.txt}%
      [arrow inside={end=To[],opt={scale=1}}{0.25,0.5,0.75}];
  \begin{scope}[red!50!blue]
    \drawIsoclines
    \hyperbolicPoint{A}
    \stablePoint{D}
  \end{scope}
  \unstablePoint{O}
\end{tikzpicture}